\documentclass[12pt]{amsart}
\usepackage{amsfonts, amssymb, latexsym, epsfig, amsthm, enumitem}
\usepackage[usenames,dvipsnames]{xcolor}
\usepackage{graphicx,pgf,tikz}
\usetikzlibrary{decorations.markings}
\usepackage{ytableau}
\usepackage[all]{xy}
\usepackage{wrapfig}
\usepackage{mathdots}
\usepackage{ bbold }

\newlist{thmlist}{enumerate}{1}
\setlist[thmlist]{label=(\roman{thmlisti}),noitemsep}

\usepackage[colorlinks=true, pdfstartview=FitV, linkcolor=blue, citecolor=blue, urlcolor=blue]{hyperref}

\newtheorem{theorem}{Theorem}[section]

\newtheorem{lemma}[theorem]{Lemma}

\newtheorem*{claim*}{Claim}
\newtheorem{corollary}[theorem]{Corollary}
\newtheorem{Main Conjecture}[theorem]{Main Conjecture}
\newtheorem{conjecture}[theorem]{Conjecture}
\theoremstyle{remark}
\newtheorem{definition}[theorem]{Definition}
\newtheorem{examplex}[theorem]{Example}
\newenvironment{example}
  {\pushQED{\qed}\examplex}
  {\popQED\endexamplex}

\newtheorem{remark}[theorem]{Remark}

\theoremstyle{plain}

\usepackage{enumitem}

\DeclareMathOperator{\Ess}{Ess}

\DeclareMathOperator{\rank}{rank}
\DeclareMathOperator{\Spec}{Spec}
\DeclareMathOperator{\Dom}{Dom}

\newcommand{\hgt}{{\rm{ht }}}

\newcommand{\CC}{\mathbb{C}}

\begin{document}
\pagestyle{plain}

\title{Diagonal degenerations of matrix Schubert varieties}

\author{Patricia Klein}
\address{Texas A\&M University, Department of Mathematics, College Station, TX, USA.}
\email[Patricia Klein]{pjklein@tamu.edu}


\maketitle

\begin{abstract}
Knutson and Miller (2005) established a connection between the anti-diagonal Gr\"obner degenerations of matrix Schubert varieties and the pre-existing combinatorics of pipe dreams.  They used this correspondence to give a geometrically-natural explanation for the appearance of the combinatorially-defined Schubert polynomials as representatives of Schubert classes.  Recently, Hamaker, Pechenik, and Weigandt (2022) proposed a similar connection between diagonal degenerations of matrix Schubert varieties and bumpless pipe dreams, newer combinatorial objects introduced by Lam, Lee, and Shimozono (2021).  Hamaker, Pechenik, and Weigandt described new generating sets of the defining ideals of matrix Schubert varieties and conjectured a characterization of permutations for which these generating sets form diagonal Gr\"obner bases.  They proved special cases of this conjecture and described diagonal degenerations of matrix Schubert varieties in terms of bumpless pipe dreams in these cases.  The purpose of this paper is to prove the conjecture in full generality.  The proof uses a connection between liaison and geometric vertex decomposition established in earlier work with Rajchgot (2021).
\end{abstract}

\tableofcontents

\section{Introduction}

Schubert polynomials, introduced by Lascoux and Sch\"utzenberger \cite{LS82} based on the work of Bernstein, Gel'fand, and Gel'fand \cite{BGG73}, give combinatorially-natural representatives of Schubert classes in the cohomology ring of the complete flag variety.  Matrix Schubert varieties, defined by Fulton \cite{Ful92}, are generalized determinantal (affine) varieties corresponding to a permutation $w \in S_n$.  A key insight of Knutson and Miller \cite{KM05} is that the combinatorics of Schubert polynomials is naturally reflected in the geometry of the initial ideals of the defining ideals $I_w$ of matrix Schubert varieties $X_w$ under anti-diagonal degeneration (i.e., Gr\"obner degeneration under a term order in which the leading term of the determinant of a generic matrix is the product of the entries along the anti-diagonal).  In particular, Knutson and Miller were able to identify irreducible components of anti-diagonal initial schemes with the \emph{pipe dreams} that had arisen in earlier combinatorial study of Schubert polynomials \cite{BB93, FK94}.  They were also able to give a geometric explanation for the positivity of coefficients of Schubert polynomials, a fact not obvious from their recursive definition, and show that the multidegrees of $X_w$ give the torus-equivariant cohomology classes of Schubert varieties.

Later, Knutson, Miller, and Yong \cite{KMY09} connected the geometry of diagonal degenerations (defined analogously to anti-diagonal degenerations) of matrix Schubert varieties corresponding to \emph{vexillary} permutations to the combinatorics of flagged tableaux.  More recently, Hamaker, Pechenik, and Weigandt \cite{HPW}  proposed a diagonal Gr\"obner basis, which they call the set of CDG generators (see Subsection \ref{subsect:CDG}), for a wider class of permutations that includes the vexillary permutations.  They proved that CDG generators form a diagonal Gr\"obner basis when $w$ is what they called \emph{banner} and used that result to connect the geometry of the diagonal degenerations of $X_w$ to the \emph{bumpless pipe dreams} introduced by Lam, Lee, and Shimozono \cite{LLS} (closely related to the $6$-vertex ice model used by Lascoux \cite{Las, Wei, BS}).  With an eye towards extending their main theorem \cite[Theorem 6.4]{HPW}, Hamaker, Pechenik, and Weigandt made the following conjecture:

\begin{conjecture}\cite[Conjecture 7.1]{HPW}\label{conj:mainresult}
Let $w \in S_n$ be a permutation. The CDG generators are a diagonal Gr\"obner basis for $I_w$ if and only if $w$ avoids all eight of the patterns \[
13254, 21543, 214635, 215364, 215634, 241635, 315264, 4261735.
\]
\end{conjecture}

The purpose of the present paper is to prove Conjecture \ref{conj:mainresult}, which we do as Corollaries \ref{cor:mainIf} and \ref{cor:mainOnlyif}.  An important step in our proof is an application of the author's work with Rajchgot \cite[Corollary 4.13]{KR}, which uses the connection between \emph{liaison}, whose use in studying Gr\"obner bases is described in \cite{GMN13}, and \emph{geometric vertex decomposition}, introduced in \cite{KMY09}, to essentially reduce the requirements of the liaison-theoretic approach to a check on one ideal containment.  

The liaison-theoretic approach to studying Gr\"obner bases is the subject of \cite{GMN13}.  In it, Gorla, Migliore, and Nagel show that the module isomorphisms making up a particular kind of pair of Gorenstein links, called an \emph{elementary $G$-biliaison}, naturally encode information about the HIlbert function that facilitates the establishment of Gr\"obner bases.  The examples in \cite{GMN13} are of various generalized determinantal ideals, and the framework they introduce is the foundation of the present paper.  

Gorla, Migliore, and Nagel's \cite{GMN13} work builds on a long history of commutative algebraic inquiry into determinantal ideals, especially since Hochster and Eagon's \cite{HE} results on determinantal ideals and Cohen--Macaulayness.  Abhyankar \cite{Abh88} studied special cases of what are now known as one-sided ladder determinantal varieties in connection to Young tableaux and Schubert varieties in flag manifolds and showed those special cases to be irreducible.  Narasimhan \cite{Nar86} established Gr\"obner bases for a more general class of one-sided and two-sided ladder determinantal ideals and used this result to show that all of the varieties in the class he studied are irreducible, extending Abhyankar's result.  Gonciuliea and Miller \cite{GM00}  extended the Gr\"obner basis results to allow sizes of minors to vary in various regions within a ladder, and Gorla \cite{Gor07} showed in full generality that the natural generators of a two-sided mixed ladder determinantal ideal form a Gr\"obner basis.  Herzog and Trung \cite{HT92} gave analogous results for cogenerated and Pfaffian ideals and used them to give an elegant formula the multiplicities of the corresponding quotient rings.  See also \cite{DNS11} for related results.  Bruns and Conca \cite{BC98} gave Gr\"obner bases for powers of determinantal ideals, which they used to show the Cohen--Macaulayness of Rees algebras associated to determinantal ideals.  In \cite{Gor08}, Gorla gave a quite substantial generalization of Gaeta's theorem.  See also, for example, \cite{Bru91,  CH97, Con98, BC01,BRW05}.

The appearance of pattern avoidance in determining when CDG generators form a Gr\"obner basis is natural in light of similar results in the Schubert literature.  For example, pattern avoidance has previously been seen to govern the singularity \cite{LS90} and Gorenstein property \cite{WY06} of Schubert varieties as well as when the \emph{Fulton generators} (see Subsection \ref{subsect:MSv}) of $I_w$ constitute a diagonal Gr\"obner basis \cite{KMY09}.  (Again, we refer to \cite{Gor07} for the first proof of one direction of this last result in the language of mixed ladder determinantal varieties.)  For a survey of results in this vein, see \cite{AB16}.

In \cite{HPW}, the authors note that Conjecture \ref{conj:mainresult} implies the following conjecture by the work of \cite{FMSD}: \begin{conjecture}\cite[Conjecture 7.2]{HPW}\label{conj:Schub}
If the (single) Schubert polynomial of $w \in S_n$ is a multiplicity-free sum of monomials, then the CDG generators of $I_w$ are a diagonal Gr\"obner basis.
\end{conjecture}

\noindent We refer the reader to \cite{HPW, Wei} for a more information on Schubert polynomials and bumpless pipe dreams.

\bigskip

\textbf{The structure of this paper:} Section \ref{sect:prelim} is devoted to preliminaries on matrix Schubert varieties and CDG generators.  In Section \ref{sect:backward}, we prove the backward direction of Conjecture \ref{conj:mainresult}, and, in Section \ref{sect:forward}, we prove the forward direction.  Finally, in Section \ref{sect:intuition}, we use geometric vertex decomposition to give some intuition on what unifies the eight non-CDG permutations listed in Conjecture \ref{conj:mainresult}.

\bigskip

\textbf{Acknowledgements:} The author thanks Zach Hamaker, Oliver Pechenik, and Anna Weigandt for helpful conversations and for graciously sharing their \LaTeX code for Rothe diagrams.  She is also grateful to Jenna Rajchgot for many very valuable conversations both directly concerning this paper and also on related material.   She thanks all four for comments on an earlier draft of this document.  The author additionally thanks the anonymous referees for their careful reading and feedback, which greatly improved this document.

\section{Preliminaries}\label{sect:prelim}

In this section, we review the basics of matrix Schubert varieties as well as the CDG generators introduced in \cite{HPW}.  For a more detailed introduction to matrix Schubert varieties, we refer the reader to \cite[Chapter 10]{Ful96}.  For basic properties of and standard terminology for Gr\"obner bases, we refer the reader to \cite{Eis95}.  

\subsection{Matrix Schubert varieties}\label{subsect:MSv}

We begin by describing how each permutation is associated to the affine variety called a matrix Schubert variety.  Throughout this paper, we will take  $[n] = \{1,2,\dots,n\}$  for any $n \geq 1$ and let $S_n$ denote the symmetric group on $[n]$.  Each permutation $w \in S_n$ is is a bijection $w:[n] \rightarrow [n]$, which we will record in its one-line notation $w = w_1 w_2 \dots w_n$ where $w_i = w(i)$.  To every $w \in S_n$, we associate a \emph{Rothe diagram} $D_w$, defined as follows:
\[
D_w = \{(i,j) \in [n] \times [n]: w(i) > j, w^{-1}(j) > i\}.
\]
A Rothe diagram has the following visualization: In an $n \times n$ grid, place a $\bullet$ in position $(i,w_i)$ for each $i \in [n]$, and draw a line down from each $\bullet$ to the bottom of the grid and a line to the right from each $\bullet$ to the side of the grid.  Then $D_w$ is the set of boxes in the grid without a $\bullet$ in them or a line through them.
For example, $D_{315642}$ is the set $\{(1,1),(1,2),(3,2),(3,4),(4,2),(4,4),(5,2)\}$ and corresponds to the visualization below, in which the elements of $D_{315642}$ appear in gray and will be referred to as the \emph{boxes} of $w$:
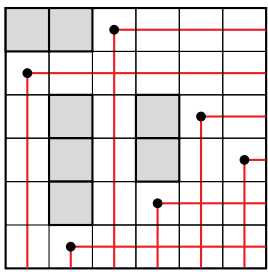
\[
  \begin{tikzpicture}[x=1.5em,y=1.5em]
      \draw[color=black, thick](0,1)rectangle(6,7);
     \filldraw[color=black, fill=gray!30, thick](0,6)rectangle(1,7);
      \filldraw[color=black, fill=gray!30, thick](1,6)rectangle(2,7);
      \filldraw[color=black, fill=gray!30, thick](1,4)rectangle(2,5);
      \filldraw[color=black, fill=gray!30, thick](3,4)rectangle(4,5);
       \filldraw[color=black, fill=gray!30, thick](1,3)rectangle(2,4);
      \filldraw[color=black, fill=gray!30, thick](3,3)rectangle(4,4);
      \filldraw[color=black, fill=gray!30, thick](1,2)rectangle(2,3);      
     \draw[thick,color=red] (6,6.5)--(2.5,6.5)--(2.5,1);
     \draw[thick,color=red] (6,5.5)--(0.5,5.5)--(0.5,1);
     \draw[thick,color=red] (6,4.5)--(4.5,4.5)--(4.5,1);
     \draw[thick,color=red] (6,3.5)--(5.5,3.5)--(5.5,1);
     \draw[thick,color=red] (6,2.5)--(3.5,2.5)--(3.5,1);
     \draw[thick,color=red] (6,1.5)--(1.5,1.5)--(1.5,1);
     \filldraw [black](2.5,6.5)circle(.1);
     \filldraw [black](0.5,5.5)circle(.1);
      \filldraw [black](4.5,4.5)circle(.1);
     \filldraw [black](5.5,3.5)circle(.1);
     \filldraw [black](3.5,2.5)circle(.1);
      \filldraw [black](1.5,1.5)circle(.1);
      \draw[color=black](2,6)rectangle(3,7);
      \draw[color=black](3,6)rectangle(4,7); 
      \draw[color=black](4,6)rectangle(5,7);
      \draw[color=black](5,6)rectangle(6,7); 
      \draw[color=black](0,5)rectangle(1,6);
      \draw[color=black](1,5)rectangle(2,6);  
      \draw[color=black](2,5)rectangle(3,6);
      \draw[color=black](3,5)rectangle(4,6); 
      \draw[color=black](4,5)rectangle(5,6);
      \draw[color=black](5,5)rectangle(6,6); 
     \draw[color=black](0,4)rectangle(1,5); 
      \draw[color=black](2,4)rectangle(3,5);
      \draw[color=black](4,4)rectangle(5,5);
      \draw[color=black](5,4)rectangle(6,5); 
      \draw[color=black](0,3)rectangle(1,4); 
      \draw[color=black](2,3)rectangle(3,4);
      \draw[color=black](4,3)rectangle(5,4);
      \draw[color=black](5,3)rectangle(6,4); 
      \draw[color=black](0,2)rectangle(1,3); 
      \draw[color=black](2,2)rectangle(3,3);
      \draw[color=black](3,2)rectangle(4,3);
      \draw[color=black](4,2)rectangle(5,3);
      \draw[color=black](5,2)rectangle(6,3); 
      \draw[color=black](1,1)rectangle(2,2);  
      \draw[color=black](2,1)rectangle(3,2);
      \draw[color=black](3,1)rectangle(4,2); 
      \draw[color=black](4,1)rectangle(5,2);
      \draw[color=black](5,1)rectangle(6,2); 
     \end{tikzpicture}.
\]

The \emph{Coxeter length} of the permutation $w$ is equal to its inversion number, i.e., 
\[| \{ (i, j) \mid i < j, w_i > w_j \}|,
\] which is in turn equal to $|D_w|$.  For example, the Coxeter length of $315642$ is $7$, easily read off as the number of gray boxes in the diagram above.

\begin{definition}
Fix a permutation $w = w_1 \cdots w_n \in S_n$ and a permutation $v  = v_1\ldots v_k \in S_k$ with $k \leq n$.  If there is some substring $w_{i_1} \cdots w_{i_k}$ of $w$ satisfying $w_{i_j} < w_{i_\ell}$ exactly when $v_j < v_\ell$, we say that $w$ \emph{contains} $v$.  Otherwise, we say that $w$ \emph{avoids} $v$.
\end{definition}

For example, $w = 13254$ contains $v = 2143$ with $3254$ the substring of $w$ realizing the containment, but $w$ does not contain $v' = 3214$.  Notice that if $w_{i_1} \cdots w_{i_k}$ satisfies $w_{i_j} < w_{i_\ell}$ exactly when $v_j < v_\ell$, then the Rothe diagram of $v$ can be obtained from that of $w$ by restricting to the rows $i_1, \ldots, i_k$ and columns $w_{i_1}, \ldots, w_{i_k}$ in the $[n] \times [n]$ grid giving the visualization of $D_w$.  

By restricting to the maximally southeast boxes of connected components of $D_w$, we define the \emph{essential set} of $w$: 
\[
\Ess(w) = \{(i,j) \in D_w \mid (i+1,j), (i,j+1) \notin D_w\}.
\]
In the example above, $\Ess(315642)=\{(1,2),(4,4), (5,2)\}$.  Borrowing a term from the literature on ladder determinantal varieties, if $(i,j) \in \Ess(w)$ and there is no $(i'j') \in \Ess(w)$ with $i \leq i'$, $j \leq j'$, and $(i,j) \neq (i',j')$, we will say that $(i,j)$ is a \emph{lower outside corner} of $D_w$.  In the case of $w = 315642$, $(4,4)$ and $(5,2)$ are lower outside corners, but $(1,2)$ is not.  

To every permutation $w \in S_n$, we associate a \emph{rank function} $\rank_w : [n] \times [n] \to \mathbb{Z}$, where 
\[
\rank_w(i,j) = |\{k \leq i \mid w(k) \leq j\}|,
\] and the rank matrix $M_w$ whose $(i,j)^{th}$ entry is $\rank_w(i,j)$.  Visually, we assign to every square $(i,j)$ in the $[n] \times [n]$ grid underlying the Rothe diagram of $w$ the number of $\bullet$s weakly northwest of $(i,j)$.  In our running example $w = 315642$, one may find it helpful to record the information as  \definecolor{burntorange}{rgb}{0.8, 0.33, 0.0}
\[
\begin{minipage}{0.4\textwidth}\
\[
\hspace{1cm} M_w = \begin{bmatrix}
\boxed{0} & {\color{burntorange}\boxed{0}} & 1 & 1 & 1 & 1\\
1 & 1 & 2 & 2 & 2 & 2\\
1 & \boxed{1} & 2 & \boxed{2} & 3 & 3\\
1 & \boxed{1} & 2 & {\color{burntorange}\boxed{2}} & 3 & 4\\
1 & {\color{burntorange}\boxed{1}} & 2 & 3 & 4 & 5\\
1 & 2 & 3 & 4 & 5 & 6\\
\end{bmatrix} 
\]

\end{minipage} \hspace{1cm} \mbox{ or as } \hspace{0.6cm} \begin{minipage}{.4\textwidth}
\hspace{0.5cm} 
  \begin{tikzpicture}[x=1.5em,y=1.5em]
      \draw[color=black, thick](0,1)rectangle(6,7);
     \filldraw[color=black, fill=gray!30, thick](0,6)rectangle(1,7);
      \filldraw[color=black, fill=gray!30, thick](1,6)rectangle(2,7);
      \filldraw[color=black, fill=gray!30, thick](1,4)rectangle(2,5);
      \filldraw[color=black, fill=gray!30, thick](3,4)rectangle(4,5);
       \filldraw[color=black, fill=gray!30, thick](1,3)rectangle(2,4);
      \filldraw[color=black, fill=gray!30, thick](3,3)rectangle(4,4);
      \filldraw[color=black, fill=gray!30, thick](1,2)rectangle(2,3);      
     \draw[thick,color=red] (6,6.5)--(2.5,6.5)--(2.5,1);
     \draw[thick,color=red] (6,5.5)--(0.5,5.5)--(0.5,1);
     \draw[thick,color=red] (6,4.5)--(4.5,4.5)--(4.5,1);
     \draw[thick,color=red] (6,3.5)--(5.5,3.5)--(5.5,1);
     \draw[thick,color=red] (6,2.5)--(3.5,2.5)--(3.5,1);
     \draw[thick,color=red] (6,1.5)--(1.5,1.5)--(1.5,1);
     \filldraw [black](2.5,6.5)circle(.1);
     \filldraw [black](0.5,5.5)circle(.1);
      \filldraw [black](4.5,4.5)circle(.1);
     \filldraw [black](5.5,3.5)circle(.1);
     \filldraw [black](3.5,2.5)circle(.1);
      \filldraw [black](1.5,1.5)circle(.1);
      \draw[color=black](2,6)rectangle(3,7);
      \draw[color=black](3,6)rectangle(4,7); 
      \draw[color=black](4,6)rectangle(5,7);
      \draw[color=black](5,6)rectangle(6,7); 
      \draw[color=black](0,5)rectangle(1,6);
      \draw[color=black](1,5)rectangle(2,6);  
      \draw[color=black](2,5)rectangle(3,6);
      \draw[color=black](3,5)rectangle(4,6); 
      \draw[color=black](4,5)rectangle(5,6);
      \draw[color=black](5,5)rectangle(6,6); 
     \draw[color=black](0,4)rectangle(1,5); 
      \draw[color=black](2,4)rectangle(3,5);
      \draw[color=black](4,4)rectangle(5,5);
      \draw[color=black](5,4)rectangle(6,5); 
      \draw[color=black](0,3)rectangle(1,4); 
      \draw[color=black](2,3)rectangle(3,4);
      \draw[color=black](4,3)rectangle(5,4);
      \draw[color=black](5,3)rectangle(6,4); 
      \draw[color=black](0,2)rectangle(1,3); 
      \draw[color=black](2,2)rectangle(3,3);
      \draw[color=black](3,2)rectangle(4,3);
      \draw[color=black](4,2)rectangle(5,3);
      \draw[color=black](5,2)rectangle(6,3); 
      \draw[color=black](1,1)rectangle(2,2);  
      \draw[color=black](2,1)rectangle(3,2);
      \draw[color=black](3,1)rectangle(4,2); 
      \draw[color=black](4,1)rectangle(5,2);
      \draw[color=black](5,1)rectangle(6,2); 
     \node at (0.5,6.5) {$0$};
      \node at (1.5,6.5) {$0$};
      \node at (1.5,4.5) {$1$};
      \node at (1.5,3.5) {$1$};
      \node at (1.5,2.5) {$1$};
       \node at (3.5,4.5) {$2$};
       \node at (3.5,3.5) {$2$};
     \end{tikzpicture}.
  \end{minipage}   
       \]
Note that the transpose of rank matrix of $w$ and transpose of the Rothe diagram of $w$ correspond to the rank matrix and Rothe diagram of $w^{-1}$, respectively, because the inverse of a permutation matrix is its transpose.  We have recorded elements of $D_w$ in the rank matrix $M_w$ as boxes and colored the boxes of the essential set orange in anticipation of our discussion of Fulton generators, below.

Let $\mbox{Mat}_{n,n}$ denote the affine $n^2$-space of complex $n \times n$ matrices.  Given $N \in \mbox{Mat}_{n,n}$ and subsets $A,B \subseteq [n]$, let $N_{A,B}$ be the submatrix of $N$ determined by the rows whose indices are elements of $A$ and the columns whose indices are elements of $B$.  Then the \emph{matrix Schubert variety} of $w \in S_n$ is the affine variety \[
X_w = \left\{Z \in \mbox{Mat}_{n,n}\mid \rank{Z_{[i],[j]}} \leq \rank_w(i,j)\ \mbox{for all}\ (i,j) \in [n] \times [n] \right\}.
\]

Let $Z = (z_{i,j})_{(i,j) \in [n] \times [n]}$ be a matrix of distinct indeterminates and $R = \mathbb{C}[Z]$ so that $\Spec(R)$ is identified with $\mbox{Mat}_{n,n}$.  The \emph{Schubert determinantal ideal} of $w$ is 
\[
I_w = ( (\rank_w(i,j)+1)\text{-minors in } Z_{[i],[j]} \mid (i,j) \in [n] \times [n] ) \subseteq R.
\]  This naive generating set will typically include a good deal of redundancy, and so we will more often consider the smaller set of \emph{Fulton generators} of $I_w$: \[
	\{(\rank_w(i,j)+1)\text{-minors in } Z_{[i],[j]} \mid (i,j) \in \Ess(w) \}.
\]  Fulton showed that the Fulton generators indeed generate $I_w$ \cite[Lemma 3.10]{Ful92}, that $I_w$ is prime, and, in particular, that $X_w \cong \Spec(R/I_w)$ as reduced schemes  \cite[Proposition 3.3]{Ful92}.  The height of the ideal $I_w$ (equivalently, codimension of $\Spec(R/I_w)$ in $\Spec(R)$) is equal to the Coxeter length of $w$ \cite[Proposition 3.3]{Ful92}.

In the example $w = 315642$, the Fulton generators of $I_w$ are $z_{1,1}$, $z_{1,2}$, the $2$-minors of \begin{center} $\begin{bmatrix}
z_{1,1} & z_{1,2}\\
z_{2,1} & z_{2,2}\\
z_{3,1} & z_{3,2}\\
z_{4,1} & z_{4,2}\\
z_{5,1} & z_{5,2}\\
\end{bmatrix},
$ and the $3$-minors of  $\begin{bmatrix}
z_{1,1} & z_{1,2} & z_{1,3} & z_{1,4}\\
z_{2,1} & z_{2,2} & z_{2,3} & z_{2,4}\\
z_{3,1} & z_{3,2} & z_{3,3} & z_{3,4}\\
z_{4,1} & z_{4,2} & z_{4,3} & z_{4,4}\\
\end{bmatrix}.
$ \end{center} 

\subsection{CDG generators}\label{subsect:CDG}

In \cite{HPW}, the authors introduce \emph{CDG generators} of defining ideals of matrix Schubert varieties.  These generators are named after Conca, De Negri, and Gorla, whose result \cite[Theorem 4.2]{CDG15} served as inspiration for the generating set used in \cite{HPW} and, in particular, for Conjecture \ref{conj:mainresult}.

\begin{definition}
Fix a permutation $w \in S_n$ and an $n \times n$ matrix $Z = (z_{i,j})_{(i,j) \in [n] \times [n]}$ of distinct indeterminates.  Let $\Dom(w) = \{(i,j) \in D_w \mid \rank_w(i,j) = 0\}$, and call $\Dom(w)$ the \emph{dominant part} of the Rothe diagram $D_w$.  From $Z$, form the matrix $Z'$ by replacing $z_{i,j}$ by $0$ whenever $(i,j) \in \Dom(w)$.  Set \[
	G'_w = \{(\rank_w(i,j)+1)\text{-minors in } Z'_{[i],[j]} \mid (i,j) \in \Ess(w)\setminus \Dom(w) \},
\] and $G_w = G'_w \cup \{z_{i,j} \mid (i,j) \in \Dom(w)\}$.  We call $G_w$ the set of \emph{CDG generators} of $I_w$.
\end{definition}

\begin{example}
If $w = 315642$ the CDG generators of $I_w$ are $z_{1,1}$, $z_{1,2}$, the $2$-minors of \begin{center} $\begin{bmatrix}
0 & 0\\
z_{2,1} & z_{2,2}\\
z_{3,1} & z_{3,2}\\
z_{4,1} & z_{4,2}\\
z_{5,1} & z_{5,2}\\
\end{bmatrix},
$ and the $3$-minors of  $\begin{bmatrix}
0 & 0 & z_{1,3} & z_{1,4}\\
z_{2,1} & z_{2,2} & z_{2,3} & z_{2,4}\\
z_{3,1} & z_{3,2} & z_{3,3} & z_{3,4}\\
z_{4,1} & z_{4,2} & z_{4,3} & z_{4,4}\\
\end{bmatrix}.
$ \end{center}
\end{example}

\noindent Notice that $\Dom(w) = \emptyset$ if and only if $w_1 = 1$, in which case the CDG generators and the Fulton generators coincide.  

\subsection{Gr\"obner bases}\label{subsect:Grobner}

Let $S=\mathbb{C}[x_1, \ldots, x_d]$.  A \emph{term order} $<$ on $S$ is a total order on the monic monomials of $S$ so that $1 \leq \mu$ for every monomial $\mu$ of $S$ and so that, for all monomials $\mu_1, \mu_2$, and $\nu$, $\mu_1<\mu_2$ implies $\mu_1 \nu < \mu_2 \nu$.  Let $f = \sum_{i=1}^k c_i \mu_i \in S$ where $c_i \in \CC$ and the $\mu_i$ are monic monomials (written in the usual way so that $\mu_i \neq \mu_j$ whenever $i \neq j$ and no $c_i$ is  $0$).  Fix $i$ so that $c_i \mu_i >c_j \mu_j$ whenever $i \neq j$, and define the \emph{leading term} of $f$ to be $LT(f)=c_i \mu_i$.  For an ideal $I$ of $S$, define the \emph{initial ideal} of $I$ to be $LT(I) = (LT(f) \mid f \in I)$.  A generating set $\mathcal{G}$ of $I$ is called a \emph{Gr\"obner basis} if $LT(I) = (LT(g) \mid g \in \mathcal{G})$.  For a detailed introduction to the theory of Gr\"obner bases, including Buchberger's algorithm, we refer the reader to \cite[Chapter 15]{Eis95}.

\begin{definition}
When the set of CDG generators forms a Gr\"obner basis for the Schubert determinantal ideal $I_w$ under every diagonal term order, we will say that the permutation \emph{$w$ is CDG}.  
\end{definition}

\bigskip

\section{Rothe diagrams of CDG permutations}\label{sect:backward}

\subsection{Obstructions to being CDG}

We begin this section by describing in terms of the Rothe diagram $D_w$ conditions that prevent $w$ from being CDG.  In Subsection \ref{subsect:backward}, we will show that when $D_w$ does not satisfy these conditions, $w$ is necessarily CDG.

Before we begin, we note that the visualization of the Rothe diagram of $214635$ is obtained from that of $215364$ by transposition.  The same is true of $315264$ and $241635$.  The visualizations of the Rothe diagrams of the remaining permutations listed in Conjecture \ref{conj:mainresult} are self transpose.   This symmetry will allow us to consolidate some of our case work below.  We understand the cardinal directions in reference to $D_w$ in terms of its visualization.  We say, for example, that $(i',j')$ is ``strictly southeast" of $(i,j)$ to mean that both $i'>i$ and also $j'>j$, or that $(i',j')$ is ``strictly south and weakly east" of $(i,j)$ to mean $i'>i$ and also $j' \geq j$.

\newpage

\begin{definition}
The permutation $w$ has an \emph{obstruction} of \begin{itemize}
\item \emph{Type 1} if there is some $(r,s) \in \Dom(w) \cap \Ess(w)$ and two distinct entries $(i,j)$ and $(i',j')$ of $D_w$ strictly southeast of $(r,s)$ with $i' \neq i$ and $j' \neq j$,
\item \emph{Type 2} if there is some $(r,s) \in \Dom(w) \cap \Ess(w)$ and two distinct entries $(i,j)$ and $(i,j')$ of $\Ess(w)$ strictly southeast of $(r,s)$ with \[
\max_k \{(k,j) \in \Dom(w)\} = \max_k \{(k,j') \in \Dom(w)\}
\] or, symmetrically, two distinct entries $(i,j)$ and $(i',j)$ of $\Ess(w)$ strictly southeast of $(r,s)$ with \[
\max_\ell \{(i,\ell) \in \Dom(w)\} = \max_\ell \{(i',\ell) \in \Dom(w)\},
\]
\item \emph{Type 3} if there are two distinct entries $(i,j)$ and $(i',j')$ of $\Ess(w) \setminus \Dom(w)$ with $(i',j')$ strictly southeast of $(i,j)$.\qedhere
\end{itemize}
\end{definition}

\begin{example}

The permutation $321654$ has an obstruction of Type $1$ with $(r,s) = (2,1)$, $(i,j) = (4,5)$ and $(i',j') = (5,4)$.  These three essential boxes are shaded in light grey.

\[
  \begin{tikzpicture}[x=1.5em,y=1.5em]
      \draw[color=black, thick](0,1)rectangle(6,7);
      \filldraw [color = lightgray] (0,5)rectangle(1,6);
      \filldraw [color = lightgray] (4,3)rectangle(5,4);
     \filldraw [color = lightgray] (3,2)rectangle(4,3);
     \draw[color=black](0,6)rectangle(1,7);
      \draw[color=black](1,6)rectangle(2,7);
      \draw[color=black](1,4)rectangle(2,5);
      \draw[color=black](3,4)rectangle(4,5);
       \draw[color=black](1,3)rectangle(2,4);
      \draw[color=black](3,3)rectangle(4,4);
      \draw[color=black](1,2)rectangle(2,3);
      \draw[color=black](0,1)rectangle(1,2);        
     \draw[thick,color=red] (6,6.5)--(2.5,6.5)--(2.5,1);
     \draw[thick,color=red] (6,5.5)--(1.5,5.5)--(1.5,1);
     \draw[thick,color=red] (6,4.5)--(0.5,4.5)--(0.5,1);
     \draw[thick,color=red] (6,3.5)--(5.5,3.5)--(5.5,1);
     \draw[thick,color=red] (6,2.5)--(4.5,2.5)--(4.5,1);
     \draw[thick,color=red] (6,1.5)--(3.5,1.5)--(3.5,1);
     \filldraw [black](2.5,6.5)circle(.1);
     \filldraw [black](0.5,4.5)circle(.1);
      \filldraw [black](4.5,2.5)circle(.1);
     \filldraw [black](5.5,3.5)circle(.1);
     \filldraw [black](3.5,1.5)circle(.1);
      \filldraw [black](1.5,5.5)circle(.1);
      \draw[color=black](2,6)rectangle(3,7);
      \draw[color=black](3,6)rectangle(4,7); 
      \draw[color=black](4,6)rectangle(5,7);
      \draw[color=black](5,6)rectangle(6,7); 
      \draw[color=black](0,5)rectangle(1,6);
      \draw[color=black](1,5)rectangle(2,6);  
      \draw[color=black](2,5)rectangle(3,6);
      \draw[color=black](3,5)rectangle(4,6); 
      \draw[color=black](4,5)rectangle(5,6);
      \draw[color=black](5,5)rectangle(6,6); 
     \draw[color=black](0,4)rectangle(1,5); 
      \draw[color=black](2,4)rectangle(3,5);
      \draw[color=black](4,4)rectangle(5,5);
      \draw[color=black](5,4)rectangle(6,5); 
      \draw[color=black](0,3)rectangle(1,4); 
      \draw[color=black](2,3)rectangle(3,4);
      \draw[color=black](4,3)rectangle(5,4);
      \draw[color=black](5,3)rectangle(6,4); 
      \draw[color=black](0,2)rectangle(1,3); 
      \draw[color=black](2,2)rectangle(3,3);
      \draw[color=black](3,2)rectangle(4,3);
      \draw[color=black](4,2)rectangle(5,3);
      \draw[color=black](5,2)rectangle(6,3); 
      \draw[color=black](1,1)rectangle(2,2);  
      \draw[color=black](2,1)rectangle(3,2);
      \draw[color=black](3,1)rectangle(4,2); 
      \draw[color=black](4,1)rectangle(5,2);
      \draw[color=black](5,1)rectangle(6,2); 
     \end{tikzpicture}.
\]

The permutation $w=4263751$ has an obstruction of Type $2$ with $(r,s) = (1,3)$, $(i,j) = (3,5)$, and $(i,j') = (5,5)$.  These three essential boxes are shaded in light grey.  Here $\max_\ell \{(3,\ell) \in \Dom(w)\} =1= \max_\ell \{(5,\ell) \in \Dom(w)\}$.  

\[
  \begin{tikzpicture}[x=1.5em,y=1.5em]
      \draw[color=black, thick](0,1)rectangle(7,8);
      \filldraw [color = lightgray] (2,7)rectangle(3,8);
      \filldraw [color = lightgray] (4,3)rectangle(5,4);
     \filldraw [color = lightgray] (4,5)rectangle(5,6);
     \draw[color=black](0,6)rectangle(1,7);
      \draw[color=black](1,6)rectangle(2,7);
      \draw[color=black](1,4)rectangle(2,5);
      \draw[color=black](3,4)rectangle(4,5);
       \draw[color=black](1,3)rectangle(2,4);
      \draw[color=black](3,3)rectangle(4,4);
      \draw[color=black](1,2)rectangle(2,3);  
       \draw[color=black](0,1)rectangle(1,2);  
      \draw[thick,color=red] (7,7.5)--(3.5,7.5)--(3.5,1);    
     \draw[thick,color=red] (7,6.5)--(1.5,6.5)--(1.5,1);
     \draw[thick,color=red] (7,5.5)--(5.5,5.5)--(5.5,1);
     \draw[thick,color=red] (7,4.5)--(2.5,4.5)--(2.5,1);
     \draw[thick,color=red] (7,3.5)--(6.5,3.5)--(6.5,1);
     \draw[thick,color=red] (7,2.5)--(4.5,2.5)--(4.5,1);
     \draw[thick,color=red] (7,1.5)--(0.5,1.5)--(0.5,1);
     \filldraw [black](0.5,1.5)circle(.1);
     \filldraw [black](2.5,4.5)circle(.1);
      \filldraw [black](4.5,2.5)circle(.1);
     \filldraw [black](5.5,5.5)circle(.1);
     \filldraw [black](3.5,7.5)circle(.1);
      \filldraw [black](1.5,6.5)circle(.1);
       \filldraw [black](6.5,3.5)circle(.1);
      \draw[color=black](2,6)rectangle(3,7);
      \draw[color=black](3,6)rectangle(4,7); 
      \draw[color=black](4,6)rectangle(5,7);
      \draw[color=black](5,6)rectangle(6,7); 
      \draw[color=black](0,5)rectangle(1,6);
      \draw[color=black](1,5)rectangle(2,6);  
      \draw[color=black](2,5)rectangle(3,6);
      \draw[color=black](3,5)rectangle(4,6); 
      \draw[color=black](4,5)rectangle(5,6);
      \draw[color=black](5,5)rectangle(6,6); 
      \draw[color=black](0,4)rectangle(1,5); 
      \draw[color=black](2,4)rectangle(3,5);
      \draw[color=black](4,4)rectangle(5,5);
      \draw[color=black](5,4)rectangle(6,5); 
      \draw[color=black](0,3)rectangle(1,4); 
      \draw[color=black](2,3)rectangle(3,4);
      \draw[color=black](4,3)rectangle(5,4);
      \draw[color=black](5,3)rectangle(6,4); 
      \draw[color=black](0,2)rectangle(1,3); 
      \draw[color=black](2,2)rectangle(3,3);
      \draw[color=black](3,2)rectangle(4,3);
      \draw[color=black](4,2)rectangle(5,3);
      \draw[color=black](5,2)rectangle(6,3); 
      \draw[color=black](1,1)rectangle(2,2);  
      \draw[color=black](2,1)rectangle(3,2);
      \draw[color=black](3,1)rectangle(4,2); 
      \draw[color=black](4,1)rectangle(5,2);
      \draw[color=black](5,1)rectangle(6,2); 
      \draw[color=black](0,7)rectangle(1,8); 
      \draw[color=black](1,7)rectangle(2,8); 
      \draw[color=black](2,7)rectangle(3,8);
      \draw[color=black](3,7)rectangle(4,8);
      \draw[color=black](4,7)rectangle(5,8);
      \draw[color=black](5,7)rectangle(6,8); 
      \draw[color=black](6,7)rectangle(7,8); 
      \draw[color=black](6,1)rectangle(7,2); 
       \draw[color=black](6,2)rectangle(7,3); 
       \draw[color=black](6,3)rectangle(7,4); 
       \draw[color=black](6,4)rectangle(7,5); 
       \draw[color=black](6,5)rectangle(7,6); 
       \draw[color=black](6,6)rectangle(7,7); 
     \end{tikzpicture}.
\]

The permutation $214365$ has an obstruction of Type $3$ with $(i,j) = (3,3)$ and $(i',j') = (5,5)$.  These two essential boxes are shaded in light grey. 

\[
  \begin{tikzpicture}[x=1.5em,y=1.5em]
      \draw[color=black, thick](0,1)rectangle(6,7);
     \filldraw [color = lightgray] (4,2)rectangle(5,3);
     \filldraw [color = lightgray] (2,4)rectangle(3,5);
     \draw[color=black](0,6)rectangle(1,7);
      \draw[color=black](1,6)rectangle(2,7);
      \draw[color=black](1,4)rectangle(2,5);
      \draw[color=black](3,4)rectangle(4,5);
       \draw[color=black](1,3)rectangle(2,4);
      \draw[color=black](3,3)rectangle(4,4);
      \draw[color=black](1,2)rectangle(2,3);
      \draw[color=black](0,1)rectangle(1,2);        
     \draw[thick,color=red] (6,6.5)--(1.5,6.5)--(1.5,1);
     \draw[thick,color=red] (6,5.5)--(0.5,5.5)--(0.5,1);
     \draw[thick,color=red] (6,4.5)--(3.5,4.5)--(3.5,1);
     \draw[thick,color=red] (6,3.5)--(2.5,3.5)--(2.5,1);
     \draw[thick,color=red] (6,2.5)--(5.5,2.5)--(5.5,1);
     \draw[thick,color=red] (6,1.5)--(4.5,1.5)--(4.5,1);
     \filldraw [black](2.5,3.5)circle(.1);
     \filldraw [black](0.5,5.5)circle(.1);
      \filldraw [black](4.5,1.5)circle(.1);
     \filldraw [black](5.5,2.5)circle(.1);
     \filldraw [black](3.5,4.5)circle(.1);
      \filldraw [black](1.5,6.5)circle(.1);
      \draw[color=black](2,6)rectangle(3,7);
      \draw[color=black](3,6)rectangle(4,7); 
      \draw[color=black](4,6)rectangle(5,7);
      \draw[color=black](5,6)rectangle(6,7); 
      \draw[color=black](0,5)rectangle(1,6);
      \draw[color=black](1,5)rectangle(2,6);  
      \draw[color=black](2,5)rectangle(3,6);
      \draw[color=black](3,5)rectangle(4,6); 
      \draw[color=black](4,5)rectangle(5,6);
      \draw[color=black](5,5)rectangle(6,6); 
     \draw[color=black](0,4)rectangle(1,5); 
      \draw[color=black](2,4)rectangle(3,5);
      \draw[color=black](4,4)rectangle(5,5);
      \draw[color=black](5,4)rectangle(6,5); 
      \draw[color=black](0,3)rectangle(1,4); 
      \draw[color=black](2,3)rectangle(3,4);
      \draw[color=black](4,3)rectangle(5,4);
      \draw[color=black](5,3)rectangle(6,4); 
      \draw[color=black](0,2)rectangle(1,3); 
      \draw[color=black](2,2)rectangle(3,3);
      \draw[color=black](3,2)rectangle(4,3);
      \draw[color=black](4,2)rectangle(5,3);
      \draw[color=black](5,2)rectangle(6,3); 
      \draw[color=black](1,1)rectangle(2,2);  
      \draw[color=black](2,1)rectangle(3,2);
      \draw[color=black](3,1)rectangle(4,2); 
      \draw[color=black](4,1)rectangle(5,2);
      \draw[color=black](5,1)rectangle(6,2); 
     \end{tikzpicture}.
\]
\end{example}

\begin{lemma}\label{lem:obstruction1} 
If the permutation $w \in S_n$ has an obstruction of Type $1$, then $w$ contains $21543$, $215634$, $214635$, $215364$, or $13254$.
\end{lemma}

An example illustrating some of the cases involved in the proof of Lemma \ref{lem:obstruction1} appears below the proof itself as Example \ref{ex:casesExample}.

\begin{proof}
Fix a permutation $w \in S_n$ that has an obstruction of Type $1$, and fix entries $(r,s)$, $(i,j)$, and $(i',j')$ as in the definition of an obstruction of Type $1$.  Consider the visualization of the Rothe diagram $D_w$.

Label the $\bullet$ in the column $s+1$ with $(a,w_a)$ and the $\bullet$ in row $r+1$ with $(b,w_b)$.  Notice $a<b$ and $w_a>w_b$.   Because $(i,j)$ and $(i',j')$ are strictly southest of $(r,s)$, both $(i,j)$ and $(i',j')$ must be south of row $b$ and east of column $w_a$.  We consider two orientations of $(i,j)$ and $(i',j')$.  Without loss of generality, assume $i<i'$.

First, if $(i,j)$ is strictly northeast of $(i',j')$, then we label the $\bullet$ in row $i$ with $(c,w_c)$ and the $\bullet$ in column $j'$ with $(d,w_d)$.  Now $a<b<c<d$ and $w_b<w_a<w_d<w_c$.  (Here $c=i$ and $w_d = j'$.  Similar renamings will occur below.)  If there is any $\bullet$ in any row strictly between $c$ and $d$ whose column index is strictly between $w_d$ and $w_c$, choose one and name it $(e,w_e)$.  Then we will have $a<b<c<e<d$ and $w_b<w_a<w_d<w_e<w_c$, which is to say that $w$ contains $21543$.  Otherwise, the $\bullet$ in row $i'$, which we call $(f,w_f)$, must be east of column $w_c$, and the $\bullet$ in column $j$, which we call $(g,w_g)$, must be south of row $d$, and so $a<b<c<f<d<g$ and $w_b<w_a<w_d<w_f<w_c<w_e$, which is to say that $w$ contains $215634$.

Alternatively, if $(i,j)$ is strictly northwest of $(i',j')$, we label the $\bullet$ in row $i'$ with $(c, w_c)$, the $\bullet$ in column $j'$ with $(d,w_d)$, the $\bullet$ in row $i$ with $(e,w_e)$, and the $\bullet$ in column $j$ with $(f,w_f)$.  If $w_e>w_c$, then $a<b<e<c<d$ while $w_b<w_a<w_d<w_c<w_e$, so $w$ contains $21543$.  Similarly, if $f>d$, then $w$ contains $21543$.  Hence, we may now assume that either $w_e<w_d<w_c$ or $w_d<w_e<w_c$ and either $f<c<d$ or $c<f<d$.  Each of these four possibilities require the containment of $215634$, $215364$, $214635$, or $13254$ (ignoring $(b,w_b)$ in case $w_e<w_d<w_c$ and $f<c<d$ and by using all six dots in the other three cases).
\end{proof}

\begin{example}\label{ex:casesExample}
We give an illustration of the case of $(i,j)$ strictly northeast of $(i',j')$ to demonstrate the process of considering allowable regions of the visualization of $D_w$ for $\bullet$s we know must exist but whose locations are unknown.  Either at least one the $\bullet$s in column $j$ or row $i'$ falls in Region I (in blue), or both fall in Region II (in green).  If the former, then $w$ contains $21543$, and, if the latter, then $w$ contains $215634$.
\definecolor{electricblue}{rgb}{0.49, 0.98, 1.0}
\definecolor{palegreen}{rgb}{0.6, 0.98, 0.6}
\begin{center}
\[
  \begin{tikzpicture}[x=2em,y=2em]
      \draw[color=black, thick](0,1)rectangle(7.5,8);
      \draw[color = black, thick](4.5,8)--(4.5,7)--(1.5,7)--(1.5,5)--(0.5,5)--(0.5,3)--(0,3);
     \draw[thick,color=red] (7.5,6.75)--(1.75,6.75)--(1.75,1);
     \draw[thick,color=red] (7.5,4.75)--(0.75,4.75)--(0.75,1);
     \filldraw [black](1.75,6.75)circle(.1);
     \filldraw [black](0.75,4.75)circle(.1);
     \node at (2.5,7.5) {$\Dom(w)$};
     \node at (.95,5.35){$(r,s)$};
     \node[draw] at (5.45,4) {$(i,j)$};
     \node[draw] at (2.75,2.5) {$(i',j')$};																				
   \draw[thick,color=red] (7.5,4)--(6.25,4)--(6.25,1);
     \draw[thick,color=red] (7.5,1.75)--(2.75,1.75)--(2.75,1);
     \filldraw [black](6.25,4)circle(.1);
     \filldraw [black](2.75,1.75)circle(.1);
     \filldraw[color = electricblue, thick](3.6,2.25)--(5.25,2.25)--(5.25,1.8)--(5.75,1.8)--(5.75,2.25)--(6.2,2.25)--(6.2,2.75)--(5.75,2.75)--(5.75,3.5)--(5.25,3.5)--(5.25,2.75)--(3.6,2.75)--(3.6,2.25);
     \node at (4.5,2.5) {$I$};
     \filldraw[color = green](5.25,1.05)rectangle(5.75,1.65);
     \filldraw[color = green](6.3,2.25)rectangle(7.45,2.75);
     \node at (6.9,2.5) {$II$};
      \node at (5.5,1.35) {$II$};
        \end{tikzpicture}.
\]
\end{center}
\end{example}

\begin{lemma}\label{lem:obstruction2}
If the permutation $w \in S_n$ contains an obstruction of Type $2$, then $w$ contains $21543$, $214635$, $241635$, $215364$, or $315264$.
\end{lemma}
\begin{proof}
Fix a permutation $w \in S_n$ that has an obstruction of Type $2$.  We will first assume that we have fixed $(i,j)$ and $(i,j')$ as in the definition of a Type $2$ obstruction with $j'<j$ and $(r,s)$ assumed to be the easternmost element of $\Dom(w) \cap \Ess(w)$ northwest of both $(i,j)$ and $(i,j')$.  As before, we consider the visualization of the Rothe diagram $D_w$.

Label the $\bullet$ in column $s+1$ with $(a,w_a)$ and the $\bullet$ in the row $r+1$ with $(b,w_b)$.  Notice $a<b$ and $w_a>w_b$.  Label the $\bullet$ in row $i$ with $(c,w_c)$, the $\bullet$ in column $j$ with $(d,w_d)$, and the $\bullet$ in column $j'$ with $(e,w_e)$.  If $e>d$, then we have $a<b<c<d<e$ and $w_b<w_a<w_e<w_d<w_c$, which is to say that $w$ contains $21543$.  Now assume $e<d$.  Because $(i',j'), (i,j) \in \Ess(w)$, there must be some $\bullet$ in column $j'+1$, which we label $(f,w_f)$, north of row $i$.  Because of the easternmost assumption on $(r,s)$ and because $\max_k \{(k,j) \in \Dom(w)\} = \max_k \{(k,j') \in \Dom(w)\}$, we must have that $f>a$.  If $f>b>a$, then $w$ contains $214635$ and, if $b>f>a$, then $w$ contains $241635$.  

A parallel argument shows that if $w$ has an obstruction of Type 2 with $i'<i$, $j' = j$, and $\max_\ell \{(i,\ell) \in \Dom(w)\} = \max_\ell \{(i',\ell) \in \Dom(w)\}$, then $w$ contains $21543$, $215364$, or $315264$.
\end{proof}

\begin{lemma}\label{lem:obstruction3}
If the permutation $w \in S_n$ has an obstruction of Type $3$, then $w$ contains $13254$, $21543$, $214635$, $215364$, $215634$, $241635$, $315264$, or $4261735$.  
\end{lemma}

\begin{proof}
Fix a permutation $w \in S_n$ that has an obstruction of Type $3$.  We fix $(i,j)$, $(i',j')$ as in the definition of an obstruction of Type $3$, and consider the visualization of the Rothe diagram of $w$.  If there is some $(r,s) \in \Dom(w) \cap \Ess(w)$ strictly northwest of $(i,j)$, then $w$ has an obstruction of Type $1$,  and so it follows from Lemma \ref{lem:obstruction1} that $w$ contains $21543$, $215634$, $215364$, $214635$, or $13254$.  Hence, we assume no such $(r,s)$ exists.  Without loss of generality, we assume that the $\bullet$ in row $1$ is west of column $j'$ and that the $\bullet$ in column $1$ is north of row $i'$.  

First suppose that the $\bullet$ in row $1$ is west of column $j$.  Then the assumption that there is no $(r,s) \in \Dom(w) \cap \Ess(w)$ strictly northwest of $(i,j)$ implies that the $\bullet$ in column $1$ is south of row $i$.  Because $(i,j) \in \Ess(w)$, there must be a $\bullet$ in column $j+1$ weakly north of row $i$.  If the $\bullet$ in column $j$ is north of row $i'$, then the $\bullet$s in row $1$ and columns $j$ and $j+1$ combine with the $\bullet$s in row $i'$ and column $j'$ to form $13254$.  If the $\bullet$ in column $j$ is south of the $\bullet$ in column $j'$, then they combine with the $\bullet$s in row $1$, column $1$, and row $i'$ to form $21543$.  And if the $\bullet$ in column $j$ is north of the $\bullet$ in column $j'$ but south of row $i'$, then all $\bullet$s described in this paragraph form $241635$. 

Alternatively, assume that the $\bullet$ in row $1$ is between columns $j$ and $j'$.  If that the $\bullet$ in column $1$ is north of row $i$, then taking transposes in the argument in the previous paragraph shows that $w$ must contain $13254$, $21543$, or $315264$.  

If the $\bullet$ in column $1$ is south of row $i$, label the $\bullet$ in row $1$ with $(a, w_a)$, any fixed $\bullet$ northwest of $(i,j)$ with $(b, w_b)$, and the $\bullet$ in column $1$ with $(c,w_c)$.  We know that there is some $\bullet$ northwest of $(i,j)$ because $(i,j) \notin \Dom(w)$.  As before, if the $\bullet$ in row $i$ is west of column $j'$ and the $\bullet$ in column $j$ is north of row $i'$, then $w$ contains $13254$.  

Suppose that the $\bullet$ in row $i$ is east of column $j'$, and label that $\bullet$ with $(d,w_d)$.  Label the $\bullet$ in row $i'$ with $(e,w_e)$, and the $\bullet$ in column $j'$ with $(f,w_f)$.  If $w_e<w_d$, then $(a,w_a)$, $(b,w_b)$, $(d,w_d)$, $(e,w_e)$, and $(f,w_f)$ form $21543$.  If $w_e>w_d$, then we consider the placement of the $\bullet$ in column $j$, which we label $(g,w_g)$.  If $g<e$, then $(a,w_a)$, $(b,w_b)$, $(d,w_d)$, $(e,w_e)$, $(f,w_f)$, and $(g,w_g)$ form $315264$.  If $e>g>f$, then all $\bullet$s $(a,w_a)$ to $(g,w_g)$, form $4261735$.  And if $f<g$, then $(b,w_b)$, $(c,w_c)$, $(e,w_e)$, $(f,w_f)$, and $(g,w_g)$ form $21543$.  

Finally, the cases in which $w_d<w_f$ (equivalently, the $\bullet$ in row $i$ west of column $j'$) and $g<e$ are achieved by taking the transpose of the configurations in the preceding paragraph.  In these cases, $w$ contains $21543$, $241635$, or $4261735$.
\end{proof}

\subsection{Permutations avoiding the specified patterns are CDG} \label{subsect:backward} The remainder of this section is devoted to the backward direction of Conjecture \ref{conj:mainresult}.  We will build up to a use of \cite[Corollary 4.13]{KR}.  We begin with some notation.

If $I_w$ is the Schubert determinantal ideal of the permutation $w \in S_n$, we will use $Z_w$ to denote the matrix obtained from an $n \times n$ matrix of indeterminates by setting $z_{i,j}$ to $0$ whenever $(i,j) \in \Dom(w)$.   If $(i,j)$ is a lower outside corner of $D_w$ and $y = z_{i,j}$, we write the CDG generators of $I_w$ as $\{yq_1+r_1, \ldots, yq_k+r_k, h_1, \ldots, h_\ell\}$ where $y$ does not divide any term of any $q_i$, $r_i$ or $h_j$.  Define $N_{y,I_w} = (h_1, \ldots, h_\ell)$ and $C_{y,I_w} = (q_1, \ldots, q_k, h_1, \ldots, h_\ell)$.  This notation mimics that in \cite{KR}.  When the CDG generators are a Gr\"obner basis of $I_w$, $C_{y,I_w}$ will be the ideal corresponding to the star and $N_{y,I_w}+(y)$ the ideal corresponding to the deletion in a geometric vertex decomposition in the sense of \cite{KMY09}.

We will call $\{q_1, \ldots, q_k, h_1, \ldots, h_\ell\}$ the \emph{CDG generators} of $C_{y,I_w}$, which is itself not typically a Schubert determinantal ideal.  With notation as above, we begin by showing that $N_{y,I_w}$ is the Schubert determinantal ideal of a permutation whose Coxeter length is smaller than that of $w$, which will be an essential component of an inductive argument.

Let $t_{a,b}$ denote the transposition $(ab) \in S_n$.

\begin{lemma}\label{lem:SchubertDel}
Suppose that $I_w$ is the Schubert determinantal ideal of the permutation $w \in S_n$ and that $(i,j)$ is a lower outside corner of $D_w$ corresponding to the variable $y = z_{i,j}$.  The ideal $N_{y,I_w}$ is the Schubert determinantal ideal of a permutation $w' \in S_n$ whose Coxeter length is strictly smaller than that of $w$.  Specifically, $w' = wt_{i,w^{-1}(j)}$, and $D_w = D_{w'} \sqcup \{(i,j)\}$.  
\end{lemma}
\begin{proof}
We claim that whenever there is some $(\rank_w(i,j)+1)$-minor with $yq+r = r$, that $r \in N_{y,I_w}$, i.e., that all of the CDG generators of $I_w$ determined only by the rank condition at $(i,j)$ involve $y$.  Fix a $(\rank_w(i,j)+1) \times (\rank_w(i,j)+1)$ submatrix $M$ of $Z_w$ so that $\det(M) = yq+r$, and suppose that $q = 0$.  Recall that the entry in row $a$ and column $b$ of $M$ is $0$ if and only if $(a,b) \in \Dom(w)$.  Then $q$ is the determinant of a $\rank_w(i,j) \times \rank_w(i,j)$ submatrix of $M$ whose anti-diagonal has an entry in row $a$ and column $b$ for some $(a,b) \in \Dom(w)$.  Assume without loss of generality that $i \geq j$.  

If $r \neq 0$, then there is some $(i',j)$ with nonvanishing $z_{i',j} \cdot q'$  summand of $r$ so that $q'$ is the determinant of a $\rank_w(i,j) \times \rank_w(i,j)$ submatrix of $M$ without a $0$ along its anti-diagonal.  Because $\Dom(w)$ forms a partition shape, if $z_{i',j}\cdot q' \neq 0$, then $z_{i'',j} \cdot q'' \neq 0$ whenever $i''<i'$ and $q''$ is the cofactor corresponding to $z_{i'',j}$ in an expansion of $yq+r$ along column $j$.  In particular, there is a unique $t$ so that no $\rank_w(i,j) \times \rank_w(i,j)$ submatrix of $M$ that excludes column $j$ and involves the final $t$ rows has a $0$ along its anti-diagonal and every $\rank_w(i,j) \times \rank_w(i,j)$ submatrix of $M$ that excludes column $j$ and also excludes one of the final $t$ rows has a $0$ along its anti-diagonal.  

The same argument can be applied to columns, and, because vanishing is determined by $0$'s along the anti-diagonal, will select the final $\rank_w(i.j)+1-t$ columns.  Hence, we may write  $r$ as the product of one $t$-minor determined by the final $t$ rows and initial $t$ columns of $M$ and one $(\rank_w(i.j)+1-t)$-minor consisting of the initial $\rank_w(i.j)+1-t$ rows and final $\rank_w(i.j)+1-t$ columns of $M$.  

Call the southeast corner of the lower block $z$ and the southeast corner of the upper block $z'$.  If there are fewer than $t$ dots northwest of $z$, then the $t$-minor that is one factor of $r$ is an element of $N_{y,I_w}$, and so $r \in  N_{y,I_w}$.  If there are $t$ or more dots northwest of $z$, then there are at most $\rank_w(i.j)-t$ dots northwest of $z'$, and so the factor of $r$ corresponding to the upper block is an element of $N_{y,I_w}$, and so $r \in N_{y,I_w}$.  Hence, $N_{y,I_w}$ is generated by the CDG generators of $I_w$ determined by the diagram boxes other than $(i,j)$.  The permutation with exactly those diagram boxes and rank conditions at those diagram boxes is $w' = wt_{i,w^{-1}(j)}$.  Then $N_{y,I_w} = I_{w'}$, the Coxeter lengh of $w'$ is one less than the Coxeter length of $w$, and $D_w = D_{w'} \sqcup \{(i,j)\}$.  
\end{proof}

\begin{example}
With notation as in Lemma \ref{lem:SchubertDel}, let $w = 215634$ and $(i,j) = (4,4)$.  If $t_{i,w^{-1}(j)} = t_{4,6}$, set $w' = wt_{4,6} = 215436$.  The Rothe diagrams of $w$ and $w'$ appear below.  We may view the Rothe diagram of $w'$ as arising from the Rothe diagram of $w$ by swapping rows $i=4$ and $w^{-1}(j)=6$.  This exchange eliminates the diagram box in position $(i,j) = (4,4)$ and leaves the rest of the diagram boxes undisturbed.

\[
 \begin{minipage}{.4\textwidth}
\hspace{0.5cm} 
  \begin{tikzpicture}[x=1.5em,y=1.5em]
      \draw[color=black, thick](0,1)rectangle(6,7);
           \node at (3,7.5) {$w = 215634$};
     \draw[color=black](0,6)rectangle(1,7);
      \draw[color=black](1,6)rectangle(2,7);
      \draw[color=black](1,4)rectangle(2,5);
      \draw[color=black](3,4)rectangle(4,5);
       \draw[color=black](1,3)rectangle(2,4);
      \draw[color=black](3,3)rectangle(4,4);
      \draw[color=black](1,2)rectangle(2,3);      
     \draw[thick,color=red] (6,6.5)--(1.5,6.5)--(1.5,1);
     \draw[thick,color=red] (6,5.5)--(0.5,5.5)--(0.5,1);
     \draw[thick,color=red] (6,4.5)--(4.5,4.5)--(4.5,1);
     \draw[thick,color=red] (6,3.5)--(5.5,3.5)--(5.5,1);
     \draw[thick,color=red] (6,2.5)--(2.5,2.5)--(2.5,1);
     \draw[thick,color=red] (6,1.5)--(3.5,1.5)--(3.5,1);
     \filldraw [black](1.5,6.5)circle(.1);
     \filldraw [black](0.5,5.5)circle(.1);
      \filldraw [black](4.5,4.5)circle(.1);
     \filldraw [black](5.5,3.5)circle(.1);
     \filldraw [black](2.5,2.5)circle(.1);
      \filldraw [black](3.5,1.5)circle(.1);
      \draw[color=black](2,6)rectangle(3,7);
      \draw[color=black](3,6)rectangle(4,7); 
      \draw[color=black](4,6)rectangle(5,7);
      \draw[color=black](5,6)rectangle(6,7); 
      \draw[color=black](0,5)rectangle(1,6);
      \draw[color=black](1,5)rectangle(2,6);  
      \draw[color=black](2,5)rectangle(3,6);
      \draw[color=black](3,5)rectangle(4,6); 
      \draw[color=black](4,5)rectangle(5,6);
      \draw[color=black](5,5)rectangle(6,6); 
     \draw[color=black](0,4)rectangle(1,5); 
      \draw[color=black](2,4)rectangle(3,5);
      \draw[color=black](4,4)rectangle(5,5);
      \draw[color=black](5,4)rectangle(6,5); 
      \draw[color=black](0,3)rectangle(1,4); 
      \draw[color=black](2,3)rectangle(3,4);
      \draw[color=black](4,3)rectangle(5,4);
      \draw[color=black](5,3)rectangle(6,4); 
      \draw[color=black](0,2)rectangle(1,3); 
      \draw[color=black](2,2)rectangle(3,3);
      \draw[color=black](3,2)rectangle(4,3);
      \draw[color=black](4,2)rectangle(5,3);
      \draw[color=black](5,2)rectangle(6,3); 
      \draw[color=black](1,1)rectangle(2,2);  
      \draw[color=black](2,1)rectangle(3,2);
      \draw[color=black](3,1)rectangle(4,2); 
      \draw[color=black](4,1)rectangle(5,2);
      \draw[color=black](5,1)rectangle(6,2); 
     \end{tikzpicture}
  \end{minipage}    \begin{minipage}{.4\textwidth}
\hspace{1cm} 
  \begin{tikzpicture}[x=1.5em,y=1.5em]
      \draw[color=black, thick](0,1)rectangle(6,7);
           \node at (3,7.5) {$w' =215436 $};
     \draw[color=black](0,6)rectangle(1,7);
      \draw[color=black](1,6)rectangle(2,7);
      \draw[color=black](1,4)rectangle(2,5);
      \draw[color=black](3,4)rectangle(4,5);
       \draw[color=black](1,3)rectangle(2,4);
      \draw[color=black](3,3)rectangle(4,4);
      \draw[color=black](1,2)rectangle(2,3);      
     \draw[thick,color=red] (6,6.5)--(1.5,6.5)--(1.5,1);
     \draw[thick,color=red] (6,5.5)--(0.5,5.5)--(0.5,1);
     \draw[thick,color=red] (6,4.5)--(4.5,4.5)--(4.5,1);
     \draw[thick,color=red] (6,3.5)--(3.5,3.5)--(3.5,1);
     \draw[thick,color=red] (6,2.5)--(2.5,2.5)--(2.5,1);
     \draw[thick,color=red] (6,1.5)--(5.5,1.5)--(5.5,1);
     \filldraw [black](1.5,6.5)circle(.1);
     \filldraw [black](0.5,5.5)circle(.1);
      \filldraw [black](4.5,4.5)circle(.1);
     \filldraw [black](3.5,3.5)circle(.1);
     \filldraw [black](2.5,2.5)circle(.1);
      \filldraw [black](5.5,1.5)circle(.1);
      \draw[color=black](2,6)rectangle(3,7);
      \draw[color=black](3,6)rectangle(4,7); 
      \draw[color=black](4,6)rectangle(5,7);
      \draw[color=black](5,6)rectangle(6,7); 
      \draw[color=black](0,5)rectangle(1,6);
      \draw[color=black](1,5)rectangle(2,6);  
      \draw[color=black](2,5)rectangle(3,6);
      \draw[color=black](3,5)rectangle(4,6); 
      \draw[color=black](4,5)rectangle(5,6);
      \draw[color=black](5,5)rectangle(6,6); 
     \draw[color=black](0,4)rectangle(1,5); 
      \draw[color=black](2,4)rectangle(3,5);
      \draw[color=black](4,4)rectangle(5,5);
      \draw[color=black](5,4)rectangle(6,5); 
      \draw[color=black](0,3)rectangle(1,4); 
      \draw[color=black](2,3)rectangle(3,4);
      \draw[color=black](4,3)rectangle(5,4);
      \draw[color=black](5,3)rectangle(6,4); 
      \draw[color=black](0,2)rectangle(1,3); 
      \draw[color=black](2,2)rectangle(3,3);
      \draw[color=black](3,2)rectangle(4,3);
      \draw[color=black](4,2)rectangle(5,3);
      \draw[color=black](5,2)rectangle(6,3); 
      \draw[color=black](1,1)rectangle(2,2);  
      \draw[color=black](2,1)rectangle(3,2);
      \draw[color=black](3,1)rectangle(4,2); 
      \draw[color=black](4,1)rectangle(5,2);
      \draw[color=black](5,1)rectangle(6,2); 
     \end{tikzpicture}.
  \end{minipage}   
       \]
\end{example}

\begin{remark}\label{rem:newEss}
With notation as in Lemma \ref{lem:SchubertDel}, it is possible that $\Ess(w') \not\subseteq \Ess(w)$, as is the case if $w = 215634$ and $(i,j) = (4,4)$.  In that case, $(3,4), (4,3) \in \Ess(w') \setminus \Ess(w)$.  

In general, the only possible elements of $\Ess(w') \setminus \Ess(w)$ are $(i-1,j)$ and $(i,j-1)$.  Indeed, suppose $(a,b) \in \Ess(w') \setminus \Ess(w)$. By the definition of $\Ess(w')$, $(a,b) \in D_{w'}$ and $(a+1,b), (a,b+1) \notin D_{w'}$.  Because $D_{w'} \subset D_w$, the assumption $(a,b) \notin \Ess(w)$ implies $(a+1,b) \in \Ess(w)$ or $(a,b+1) \in \Ess(w)$.  The fact that $D_w \setminus D_{w'} = \{(i,j)\}$ then implies that $(a+1,b) = (i,j)$ or $(a,b+1) = (i,j)$.
\end{remark}

\begin{corollary}\label{cor:GrobnerDel}
If $w$ has no obstruction of Type $1$, $2$, or $3$, $(i,j)$ is a lower outside corner of $D_w$ corresponding to the variable $y = z_{i,j}$, and $I_{w'} =N_{y,I_w}$, then $w'$ has no obstruction of Type $1$, $2$, or $3$.  
\end{corollary}

\begin{proof}
If $(i,j) \in \Dom(w)$, then, because there are no diagram boxes southeast of $(i,j)$ by assumption, $(i,j)$ cannot be one of the boxes involved in any of the obstructions.  In that case, any set of diagram boxes realizing an obstruction in $w'$ would also realize an obstruction in $w$.  Hence we assume that $(i,j) \notin \Dom(w)$, in which case $\Dom(w) = \Dom(w')$.

Suppose that $w'$ has an obstruction of Type $1$ with $(r,s) \in \Dom(w') \cap \Ess(w')$ and $(a,b)$, $(a'b') \in D_{w'}$ strictly southeast of $(r,s)$ with $a' \neq a$ and $b' \neq b$.  Because $\Dom(w') = \Dom(w)$ and $D_{w'} \subset D_w$, the diagram boxes $(r,s)$, $(a,b)$, $(a',b')$ also constitute an obstruction of Type $1$ in $w$ as well.

Next suppose that $w'$ has an obstruction of Type $2$.  Assume without loss of generality that the Type $2$ obstruction has the form  $(r,s) \in \Dom(w') \cap \Ess(w')$ and $(a,b), (a,b') \in \Ess(w')$ strictly southeast of $(r,s)$ with $\max_k \{(k,b) \in \Dom(w')\} = \max_k \{(k,b') \in \Dom(w') \}$ and $b<b'$.  In this case neither $(a,b)$ nor $(r,s)$ may be a lower outside corner of $w'$.   In particular, neither $(r,s)$ nor $(a,b)$ is equal to $(i,j)$, and so $(r,s) \in \Dom(w) \cap \Ess(w)$ and $(a,b) \in \Ess(w)$.  If $(a,b') \in \Ess(w)$ also, then $(r,s)$, $(a,b)$, and $(a,b')$ also constitute a Type $2$ obstruction in $w$.  If $(a,b') \notin \Ess(w)$, then either $(a+1,b') = (i,j) \in \Ess_w$, in which case $(r,s)$, $(a,b)$, and $(a+1,b')$ together constitute a Type $1$ obstruction, or $(a,b'+1) = (i,j) \in \Ess(w)$.  

If $(a,b'+1) = (i,j) \in \Ess(w)$, set $m = \max_k \{(k,b) \in \Dom(w')\} = \max_k \{(k,b') \in \Dom(w')\}$.  Because $\Dom(w) = \Dom(w')$, we have $m = \max_k \{(k,b) \in \Dom(w)\}$.  We claim that $m = \max_k \{(k,b'+1) \in \Dom(w)\}$.   Because $b'+1>b'$ and $\Dom(w)$ forms a partition shape, $m \geq \max_k \{(k,b'+1) \in \Dom(w)\}$.  If $m > \max_k \{(k,b'+1) \in \Dom(w)\}$, then $(m,b'+1) \notin \Dom(w)$.  But $(m,b') \in \Dom(w') = \Dom(w)$.  Therefore, if $(m,b'+1) \notin \Dom(w)$, the visualization of $D_w$ must have a $\bullet$ in column $b'+1$ weakly north of row $m$.  Clearly $m<a$ since $(a,b) \notin \Dom(w')$.  But the $\bullet$ in column $b'+1 = j$ must be south of row $a = i$ because $(i,j) \in D_w$. Hence, $m = \max_k \{(k,b'+1) \in \Dom(w)\}$ and $(r,s)$, $(a,b)$, and $(a, b'+1)$ together constitute a Type $3$ obstruction in $w$. 

Finally suppose that $w'$ has an obstruction of Types $3$.  Suppose that $(a,b), (a',b') \in \Ess(w') \setminus \Dom(w')$ with $(a',b')$ strictly southeast of $(a,b)$.  If $(a',b') \in \Ess(w)$, then $(a,b)$ and $(a',b')$ constitute a Type $3$ obstruction in $w$ also.  Otherwise, it must be that $(i,j) = (a'+1,b')$ or $(i,j) = (a',b+1)$.  In either case, $(a,b)$ and $(i,j)$ constitute a Type $3$ obstruction in $w$.
\end{proof}

Next, we will show that the CDG generators of $C_{y,I_w}$ form a Gr\"obner basis whenever $w$ has no obstruction of Type $1$, $2$, or $3$.  Before proceeding, we review some standard notation and make one new definition to help with bookkeeping during this subsection.  We will use $\deg(f)$ to denote the degree of the homogeneous polynomial $f$ and $LCM(\mu_1,\mu_2)$ to denote the least common multiple of two monomials (which will arise for us as the monic leading terms of ideal generators).  

\begin{definition}\label{def:Qdef}
Fix a permutation $w \in S_n$, lower outside corner $(i,j)$ of $D_w$ corresponding to the variable $y = z_{i,j}$ in $Z_w$, and CDG generators $\{yq_1+r_1, \ldots, yq_k+r_k, h_1, \ldots, h_\ell \}$ of $I_w$, where $y$ does not divide any $q_a$, $r_a$, or $h_b$.  Assume also that there is some $0 \leq \ell' \leq \ell$ so that all variables appearing in $h_b$ are northwest of some $z_{i',j}$ for $(i',j) \in \Ess(w)$ with $\rank_w(i',j) +1= \deg(h_b)$ or of some $z_{i,j'}$ for $(i,j') \in \Ess(w)$ with $\rank_w(i,j') +1= \deg(h_b)$ if and only if $b \leq \ell'$.  If $(1,j) \in \Dom(w)$, set $m_1 = \min\{i-p \mid (p,j) \in \Dom(w)\}$, and set $m_1 = i$ if $(1,j) \notin \Dom(w)$.  Similarly, if $(i,1) \in \Dom(w)$, set $m_2 = \min\{j-q \mid (i,q) \in \Dom(w)\}$, and set $m_2 = j$ if $(1,j) \notin \Dom(w)$. 

We form the ideal \[ Q_{y,I_w} = \begin{cases} 
      (q_1, \ldots, q_k) & \rank_w(i,j)+1 = \min\{m_1,m_2\} \\
     (q_1, \ldots, q_k, h_1, \ldots, h_{\ell'}) & \mbox{otherwise}. \\
   \end{cases}
\]
\end{definition}

Less formally, we are taking $Q_{y,I_w} = (q_1, \ldots, q_k)$ when the rank condition on $(i,j)$ is determining maximal minors in the submatrix of $Z_w$ obtained from the submatrix northwest of $z_{i,j}$ by deleting any full rows or columns of $0$'s.  Otherwise, we include also as generators of $Q_{y,I_w}$ the CDG generators determined by essential boxes in the same row or column as $z_{i,j}$.  

\begin{example}
Let $w = 351624$ and $y = z_{4,4}$.  The Rothe diagram of $w$ is \[
  \begin{tikzpicture}[x=1.5em,y=1.5em]
      \draw[color=black, thick](0,1)rectangle(6,7);
      \draw[color=black](1,4)rectangle(2,5);
      \draw[color=black](3,4)rectangle(4,5);
       \draw[color=black](1,3)rectangle(2,4);
      \draw[color=black](3,3)rectangle(4,4);
      \draw[color=black](1,2)rectangle(2,3);
      \draw[color=black](0,1)rectangle(1,2);  
     \draw[thick,color=red] (6,6.5)--(2.5,6.5)--(2.5,1);      
     \draw[thick,color=red] (6,5.5)--(4.5,5.5)--(4.5,1);
     \draw[thick,color=red] (6,4.5)--(0.5,4.5)--(0.5,1);
     \draw[thick,color=red] (6,3.5)--(5.5,3.5)--(5.5,1);
     \draw[thick,color=red] (6,2.5)--(1.5,2.5)--(1.5,1);
     \draw[thick,color=red] (6,1.5)--(3.5,1.5)--(3.5,1);
     \filldraw [black](0.5,4.5)circle(.1);
      \filldraw [black](5.5,3.5)circle(.1);
     \filldraw [black](4.5,5.5)circle(.1);
     \filldraw [black](3.5,1.5)circle(.1);
      \filldraw [black](1.5,2.5)circle(.1);
      \filldraw [black](2.5,6.5)circle(.1);
      \draw[color=black](0,6)rectangle(1,7);
      \draw[color=black](1,6)rectangle(2,7);  
      \draw[color=black](2,6)rectangle(3,7);
      \draw[color=black](3,6)rectangle(4,7); 
      \draw[color=black](4,6)rectangle(5,7);
      \draw[color=black](5,6)rectangle(6,7);
      \draw[color=black](0,5)rectangle(1,6);
      \draw[color=black](1,5)rectangle(2,6);  
      \draw[color=black](2,5)rectangle(3,6);
      \draw[color=black](3,5)rectangle(4,6); 
      \draw[color=black](4,5)rectangle(5,6);
      \draw[color=black](5,5)rectangle(6,6);
     \draw[color=black](0,4)rectangle(1,5); 
      \draw[color=black](2,4)rectangle(3,5);
      \draw[color=black](4,4)rectangle(5,5);
      \draw[color=black](5,4)rectangle(6,5);
      \draw[color=black](0,3)rectangle(1,4); 
      \draw[color=black](2,3)rectangle(3,4);
      \draw[color=black](4,3)rectangle(5,4);
      \draw[color=black](5,3)rectangle(6,4);
      \draw[color=black](0,2)rectangle(1,3); 
      \draw[color=black](2,2)rectangle(3,3);
      \draw[color=black](3,2)rectangle(4,3);
      \draw[color=black](4,2)rectangle(5,3);
      \draw[color=black](5,2)rectangle(6,3);
      \draw[color=black](1,1)rectangle(2,2);  
      \draw[color=black](2,1)rectangle(3,2);
      \draw[color=black](3,1)rectangle(4,2); 
      \draw[color=black](4,1)rectangle(5,2);
      \draw[color=black](5,1)rectangle(6,2);
     \end{tikzpicture}.
\] 
Then \[
I_w = (z_{1,1}, z_{1,2}, z_{2,1}, z_{2,2}, 2\text{-minors in } (Z_w)_{[4],[2]},  2\text{-minors in } (Z_w)_{[2],[4]}, 3\text{-minors in } (Z_w)_{[4],[4]})
\] and  \[
Q_{y,I_w} = (2\text{-minors in } (Z_w)_{[4],[2]},  2\text{-minors in } (Z_w)_{[2],[4]}, 2\text{-minors in } (Z_w)_{[3],[3]}).
\] The $2$-minors in $(Z_w)_{[2],[4]}$, for example, are included as generators of $Q_{y,I_w}$ because \[
\rank_w(4,4)+1 = 3<4 = \min\{4,4\} = \min\{m_1,m_2\}
\] and $(2,4) \in \Ess(w)$ is in the same column as $(4,4)$.  
\end{example}

We make this definition purely for technical convenience below and not out of an independent interest in $\Spec(R/Q_{y,I_w})$.  We will say that $Q_{y,I_w}$ is CDG if the generators given above form a Gr\"obner basis under any diagonal term order.  

\begin{lemma}\label{lem:GrobnerQs}
If $w \in S_n$ has no obstruction of Type $1$, $2$, or $3$ and Conjecture \ref{conj:mainresult} holds for all permutations of smaller Coxeter length than that of $w$, then either $D_w = \Dom(w)$ or there is some lower outside corner $(i,j)$ of $D_w \setminus \Dom(w)$ corresponding to the variable $y$ in $Z_w$ so that $Q_{y,I_w}$ is CDG.
\end{lemma}
\begin{proof}
Fix a permutation $w \in S_n$ that has no obstruction of Type $1$, $2$, or $3$, and assume $D_w \neq \Dom(w)$.  First suppose that $D_w$ has some lower outside corner $(i,j)\notin \Dom(w)$ corresponding to the variable $y$ in $Z_w$ satisfying $\rank_w(i,j)+1 = \min\{m_1,m_2\}$, with notation as in Definition \ref{def:Qdef}.  

Write the CDG generators of $I$ as $\{yq_1+r_1, \ldots, yq_k+r_k, h_1, \ldots, h_\ell \}$, where $y$ does not divide any $q_i$, $r_i$, or $h_j$.  As discussed in Lemma \ref{lem:SchubertDel}, every $(\rank_w(i,j)+1)$-minor involving row $i$ and column $j$ has a term divisible by $y$, and so $\{q_1, \ldots, q_k\}$ generates the ideal of $\rank_w(i,j)$-minors in the submatrix of $Z_w$ strictly northwest of $y$.  Because the $Q_{y,I}$ is an ideal of maximal minors (after possible removing full rows or columns of $0$'s), the result follows from \cite[Theorem 4.2]{CDG15} or \cite[Proposition 5.4]{Boo11}.  

Alternatively, suppose that $D_w$ has no such lower outside corner, and fix any lower outside corner of $D_w$.  Because $Q_{y,I_w}$ depends only on $D_w$ weakly northwest of $(i,j)$, we may assume that $(i,j)$ is the only lower outside corner of $D_w$ and that $(1,j), (i,1) \notin \Dom(w)$.  With this assumption, $Q_{y,I_w} +(z_{a,b} \mid (a,b) \in \Dom(w))= C_{y,I_w}$.  Because the $z_{a,b}$ with $(a,b) \in \Dom(w)$ are indeterminates that do not divide any term of any CDG generator of $Q_{y,I_w}$, it follows that $Q_{y,I_w}$ is CDG if and only if $C_{y,I_w}$ is CDG.  

If $\Dom(w) = \emptyset$, then, because $w$ has no obstruction of Type $1$, $w$ must be vexillary and so $I_w$ is CDG by \cite[Theorem 3.8]{KMY09}.   It then follows from  \cite[Theorem 2.1(a)]{KMY09} that $C_{y,I_w}$ is CDG as well.  

Now suppose $\Dom(w) \neq \emptyset$.  The assumptions that $(1,j), (i,1) \notin \Dom(w)$ imply that there is at least one element of $\Dom(w) \cap \Ess(w)$ northwest of $(i,j)$.  Choose the southernmost such element and label it $(r,s)$ and the easternmost element and label it $(r',s')$.  Because $w$ has no obstruction of Type $1$, all essential boxes of $D_w \setminus \Dom(w)$ must be in either row $i$ or column $j$.  Because $\rank_w(i,j)+1 < \min\{i,j\}$, there must be at least one essential box in row $i$ and at least one in column $j$ aside from $(i,j)$. Then because $w$ has no obstruction of Type $2$, there are no elements of $\Ess(w)$ north of row $i$ and south of row $r$ or west of column $j$ and east of column $s'$.  

We will argue directly in this case that for each pair $q_a$ and $h_b$, their $s$-polynomial has a Gr\"obner reduction by the CDG generators of $Q_{y,I_w}$.  (For an example illustrating this process, see Example \ref{ex:Spoly} below.)  Choose such a $q_a$ and $h_b$.  Because the CDG generators of $N_{y,I_w}$ form a Gr\"obner basis by induction on the Coxeter length of $w$ and Corollary \ref{cor:GrobnerDel}, we may assume that $q_a \notin (h_1, \ldots, h_\ell)$.  

Because $Q_{y,I_w}$ involves only indeterminates of $Z_w$ northwest of $z_{i,j}$, we will work within the submatrix of $Z_w$ northwest of $z_{i,j}$, which we will call $Y_w$.  Suppose that $h_b$ is determined by the essential box at $(i',j)$ for some $i'<i$.  (The case of $(i,j')$ with $j'<j$ will follow by symmetry.)  We consider two cases.  

Case 1: Suppose that $(i',j-1) \in D_w$.  Then, because $w$ has no obstruction of Type $1$, there can be no element of $\Dom(w) \cap \Ess(w)$ northwest of $(i',j)$.  In particular, $r' \geq i'$, and $h_b$ is a $(\rank_w(i',j)+1)$-minor in the submatrix of $Y_w$ formed of its final $j-s'$ columns, which is a generic matrix and which we will call $Y'_w$.  Suppose that there are $t$ columns determining $q_a$ strictly east of column $s'$ and $\rank_w(i,j)-t$ columns determining $q_a$ weakly west of column $s'$.  Express $q_a$ as a sum of products of $t$-minors and $(\rank_w(i,j)-t)$-minors corresponding to this subdivision.  For any fixed set of rows of $Y'_w$, the set of $(\rank_w(i',j)+1)$-minors weakly northwest of $(i',j)$ together with the $t$-minors strictly northwest of $(i,j)$ in the submatrix of $Y'_w$ including only those specified rows forms a Gr\"obner basis for the ideal they generate because it is a mixed ladder determinantal ideal \cite[Theorem 1.10]{Gor07}.  

Choose the $t$-minor $\varepsilon_1$ from the eastern $t$ columns determining $q_a$ and the $(\rank_w(i,j)-t)$-minor $\varepsilon_2$ from the remaining columns satisfying $LT(\varepsilon_1) \cdot LT(\varepsilon_2) = LT(q_a)$.  Because $\varepsilon_1$ and $h_b$ belong to the ideal of $t$-minors weakly north of their southernmost entries together with the $(\rank_w(i',j)+1)$-minors northwest of $(i',j)$ in $Y'_w$, their $s$-polynomial $s(\varepsilon_1,h_b)$ has a Gr\"obner reduction $s(\varepsilon_1,h_b) = \sum \alpha_c \delta_c $ by the natural generators of that ideal.  For each $\delta_c \in N_{y,I_w}$, set $\widehat{\delta_c} = LT(\varepsilon_2) \cdot \delta_c$, and for each $\delta_c$ involving some row south of $i'$, set $\widehat{\delta_c}$ to be (up to sign) the determinant of the augmentation of the matrix determining $\delta_c$ by the rows and columns determining $\varepsilon_2$ (with sign chosen so that $LT(\delta_c)$ and $LT(\widehat{\delta_c})$ share a sign).  Let $s(q_a,h_b)$ denote the $s$-polynomial of $q_a$ and $h_b$.  

We claim that $s(q_a,h_b)-\sum \alpha_c \widehat{\delta_c} \in N_{y,I}$.  It is clear that $s(q_a,h_b)-\sum \alpha_c \widehat{\delta_c}$ contains a $LT(\varepsilon_2)$-multiple of $s(\varepsilon_1,h_b)-\sum \alpha_c \delta_c$, which is $0$ because $s(\varepsilon_1,h_b)-\sum \alpha_c \delta_c$ is.  Fix any non-leading term $\mu$ of $\varepsilon_2$, and write $s(q_a,h_b)-\sum \alpha_c \widehat{\delta_c} = \mu \tilde{s}+\tilde{\tilde{s}}$ where $\mu$ does not divide any term of $\tilde{\tilde{s}}$.  For each column involved in $\varepsilon_2$, whenever a different variable from that column divides $\mu$ and $LT(\varepsilon_2)$, replace in $\sum \alpha_c \widehat{\delta_c}$ the variables in the row of the divisor of $LT(\varepsilon_2)$ with the variables in the same columns from the row of the corresponding divisor of $\mu$, which we note gives an expression of $\tilde{s}$ in terms of the natural generators of the ideal of $(\rank_w(i',j)+1)$-minors northwest of $(i',j)$, each of which is a CDG generator of $N_{y,I}$.

Now because $LT(\alpha_c\delta_c) \leq LT(s(\varepsilon_1,h_b))$ and $LT(\alpha_c \widehat{\delta_c}) = LT(\varepsilon_2) \cdot LT(\alpha_c\delta_c)$ for each $c$ and because $LT(s(q_a,h_b)) = LT(\varepsilon_2) \cdot LT(s(\varepsilon_1,h_b))$, subtracting each $\alpha_c \widehat{\delta_c}$ is a valid step in a Gr\"obner reduction of $s(q_a,h_b)$ by the generators of $Q_{y,I_w}$.  The fact that the CDG generators of $N_{y,I_w}$, each of which is a CDG generator of $Q_{y,I_w}$, form a Gr\"obner basis for the ideal they generate implies that $s(q_a,h_b)-\sum \alpha_c \widehat{\delta_c} \in N_{y,I_w}$ has a reduction in terms of those generators and so that $s(q_a,h_b)$ has a reduction by the CDG generators of $Q_{y,I_w}$.

Case 2: Suppose $(i',j-1) \notin D_w$.  Then there can be no $j''<j$ with $(i',j'') \in D_w\setminus \Dom(w)$ because $w$ has no obstruction of Type $3$.  Hence, $\rank_w(i',j)+1 = \min\{j-k \mid (i',k) \in \Dom(w)\}$, and so the $(\rank_w(i',j)+1)$-minors northwest of $(i',j)$ are the maximal minors of the submatrix of $Z_w$ northwest of $(i',j)$ after removing complete rows or columns of $0$'s.  Then the argument is similar to the first case but uses, instead of results on ladder determinantal ideals in a generic matrix, the fact that the maximal minors of matrices of indeterminates and $0$'s form a Gr\"obner basis by \cite[Theorem 4.2]{CDG15} or \cite[Proposition 5.4]{Boo11}.  
\end{proof}

\begin{example}\label{ex:Spoly}
Observe that $w = 378149256$, whose annotated Rothe diagram appears below, is an example of a CDG permutation with a unique lower outside corner  $(i,j) = (6,6)$ and $4=\rank_w(i,j)+1 \neq \min\{m_1,m_2\}=6$.  We illustrate how the reduction of the $s$-polynomial of $h_b = \begin{vmatrix} z_{2,3} & z_{2,4} \\ z_{3,3} & z_{3,4} \end{vmatrix}$ and $\begin{vmatrix} z_{2,3} & z_{2,5} \\ z_{5,3} & z_{5,5} \end{vmatrix}$ gives rise to that of $\begin{vmatrix} z_{2,3} & z_{2,4} \\ z_{3,3} & z_{3,4} \end{vmatrix}$ and $q_a = \begin{vmatrix} 0 & z_{2,3} & z_{2,5} \\ z_{4,2} & z_{4,3} & z_{4,5} \\ z_{5,2} & z_{5,3} & z_{5,5} \end{vmatrix}$ by the process described in Lemma \ref{lem:GrobnerQs}.  Using that \[
z_{5,5} { \color{burntorange}\begin{vmatrix} z_{2,3} & z_{2,4} \\ z_{3,3} & z_{3,4} \end{vmatrix}}-z_{3,4} \begin{vmatrix} z_{2,3} & z_{2,5} \\ z_{5,3} & z_{5,5} \end{vmatrix} +z_{2,4} \begin{vmatrix} z_{3,3} & z_{3,5} \\ z_{5,3} & z_{5,5} \end{vmatrix} + z_{5,3} {\color{burntorange} \begin{vmatrix} z_{2,4} & z_{2,5} \\ z_{3,4} & z_{3,5} \end{vmatrix}} = 0,
\] we construct the equation \begin{align*}
 {\color{blue}\underline{z_{4,2} z_{5,5} \begin{vmatrix} z_{2,3} & z_{2,4} \\ z_{3,3} & z_{3,4} \end{vmatrix}}}+{\color{blue}\underline{z_{3,4}} }\begin{vmatrix} 0 & \color{blue}\underline{z_{2,3}} & \color{blue}\underline{z_{2,5}} \\ \color{blue}\underline{z_{4,2}} & z_{4,3} & z_{4,5} \\  z_{5,2} & \color{blue}\underline{z_{5,3}} & \color{blue}\underline{z_{5,5}} \end{vmatrix}  &- 
  {\color{blue}\underline{z_{2,4}}} \begin{vmatrix} 0 & \color{blue}\underline{z_{3,3}} & \color{blue}\underline{z_{3,5}} \\ \color{blue}\underline{z_{4,2}} & z_{4,3} & z_{4,5} \\ z_{5,2} & \color{blue}\underline{z_{5,3}} & \color{blue}\underline{z_{5,5}} \end{vmatrix}+ {\color{blue}\underline{z_{4,2} z_{5,3} \begin{vmatrix} z_{2,4} & z_{2,5} \\ z_{3,4} & z_{3,5} \end{vmatrix} }}\\
  &= z_{5,2} \left(z_{4,5}   {\color{burntorange}\begin{vmatrix} z_{2,3} & z_{2,4} \\ z_{3,3} & z_{3,4} \end{vmatrix}} +z_{4,3} { \color{burntorange} \begin{vmatrix} z_{2,4} & z_{2,5} \\ z_{3,4} & z_{3,5} \end{vmatrix}}\right).
\end{align*} In the latter equation, we find an $z_{4,2}$-multiple of the first equation (whose terms appear in blue and also are underlined) and an $z_{5,2}$-multiple of an element easily seen to be in $N_{y,I_w}$.  We record in orange the generators of $N_{y,I_w}$ that give the inclusion of the $z_{5,2}$-multiple summand of $s(q_a,h_b)$ in $N_{y,I_w}$ and see that relation arising from the first equation.  The new relation is obtained from the first by exchanging rows $4$ and $5$.  
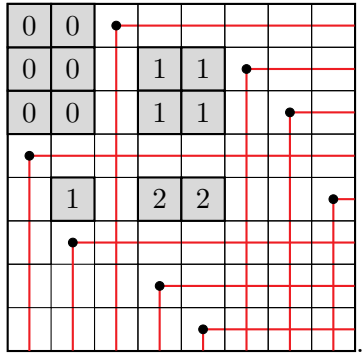
\[
\hspace{1cm} 
\hspace{0.5cm} 
  \begin{tikzpicture}[x=1.5em,y=1.5em]
      \draw[color=black, thick](0,1)rectangle(9,10);
     \filldraw[color=black, fill=gray!30, thick](0,9)rectangle(1,10);
      \filldraw[color=black, fill=gray!30, thick](1,9)rectangle(2,10);
      \filldraw[color=black, fill=gray!30, thick](0,7)rectangle(1,8);
      \filldraw[color=black, fill=gray!30, thick](0,8)rectangle(1,9);
      \filldraw[color=black, fill=gray!30, thick](1,8)rectangle(2,9);
      \filldraw[color=black, fill=gray!30, thick](1,7)rectangle(2,8);
      \filldraw[color=black, fill=gray!30, thick](1,4)rectangle(2,5);
     \filldraw[color=black, fill=gray!30, thick](1,5)rectangle(2,6);
      \filldraw[color=black, fill=gray!30, thick](5,4)rectangle(6,5);
       \draw[color=black](1,3)rectangle(2,4);
      \draw[color=black](3,3)rectangle(4,4);  
     \draw[thick,color=red] (9,9.5)--(2.5,9.5)--(2.5,1);
     \draw[thick,color=red] (9,8.5)--(6.5,8.5)--(6.5,1);
     \draw[thick,color=red] (9,5.5)--(3.5,5.5)--(3.5,1);
     \draw[thick,color=red] (9,1.5)--(5.5,1.5)--(5.5,1);
     \draw[thick,color=red] (9,7.5)--(7.5,7.5)--(7.5,1);
     \draw[thick,color=red] (9,2.5)--(4.5,2.5)--(4.5,1);
     \draw[thick,color=red] (9,3.5)--(1.5,3.5)--(1.5,1);
     \draw[thick,color=red] (9,6.5)--(0.5,6.5)--(0.5,1);
     \draw[thick,color=red] (9,4.5)--(8.5,4.5)--(8.5,1);
     \filldraw [black](2.5,9.5)circle(.1);
     \filldraw [black](6.5,8.5)circle(.1);
     \filldraw [black](0.5,6.5)circle(.1);
      \filldraw [black](5.5,1.5)circle(.1);
     \filldraw [black](7.5,7.5)circle(.1);
     \filldraw [black](4.5,2.5)circle(.1);
      \filldraw [black](1.5,3.5)circle(.1);
      \filldraw [black](3.5,5.5)circle(.1);
      \filldraw [black](8.5,4.5)circle(.1);
       \draw[color=black](2,7)rectangle(3,8);
      \filldraw[color=black, fill=gray!30, thick](3,7)rectangle(4,8); 
      \filldraw[color=black, fill=gray!30, thick](4,7)rectangle(5,8);
      \filldraw[color=black, fill=gray!30, thick](5,7)rectangle(6,8); 
      \filldraw[color=black, fill=gray!30, thick](5,8)rectangle(6,9);
      \draw[color=black](5,7)rectangle(6,8); 
     \draw[color=black](2,8)rectangle(3,9);
      \draw[color=black](3,8)rectangle(4,9); 
      \draw[color=black](4,8)rectangle(5,9);
      \draw[color=black](5,8)rectangle(6,9); 
      \draw[color=black](2,6)rectangle(3,7);
      \draw[color=black](0,1)rectangle(1,2);
      \draw[color=black](1,2)rectangle(2,3);
      \draw[color=black](3,4)rectangle(4,5);
      \draw[color=black](1,6)rectangle(2,7);
      \draw[color=black](2,6)rectangle(3,7);
      \filldraw[color=black, fill=gray!30, thick](3,8)rectangle(4,9); 
      \filldraw[color=black, fill=gray!30, thick](4,8)rectangle(5,9);
      \draw[color=black](5,6)rectangle(6,7); 
      \draw[color=black](0,5)rectangle(1,6);
      \draw[color=black](1,5)rectangle(2,6);  
      \draw[color=black](2,5)rectangle(3,6);
      \draw[color=black](3,5)rectangle(4,6); 
      \draw[color=black](4,5)rectangle(5,6);
      \draw[color=black](5,5)rectangle(6,6); 
     \draw[color=black](0,4)rectangle(1,5); 
      \draw[color=black](2,4)rectangle(3,5);
    \filldraw[color=black, fill=gray!30, thick](4,4)rectangle(5,5);
      \draw[color=black](5,4)rectangle(6,5); 
      \draw[color=black](0,3)rectangle(1,4); 
      \draw[color=black](2,3)rectangle(3,4);
      \draw[color=black](4,3)rectangle(5,4);
      \draw[color=black](5,3)rectangle(6,4); 
      \draw[color=black](0,2)rectangle(1,3); 
      \draw[color=black](2,2)rectangle(3,3);
      \draw[color=black](3,2)rectangle(4,3);
      \draw[color=black](4,2)rectangle(5,3);
      \draw[color=black](5,2)rectangle(6,3); 
      \draw[color=black](1,1)rectangle(2,2);  
      \draw[color=black](0,6)rectangle(1,7);  
      \draw[color=black](0,9)rectangle(1,10); 
      \draw[color=black](1,9)rectangle(2,10); 
      \draw[color=black](2,1)rectangle(3,2);
      \draw[color=black](2,9)rectangle(3,10); 
      \draw[color=black](3,1)rectangle(4,2); 
      \draw[color=black](3,9)rectangle(4,10); 
      \draw[color=black](4,1)rectangle(5,2);
      \draw[color=black](4,9)rectangle(5,10); 
      \draw[color=black](5,1)rectangle(6,2); 
      \draw[color=black](5,9)rectangle(6,10); 
      \draw[color=black](6,1)rectangle(7,2); 
      \draw[color=black](6,2)rectangle(7,3); 
      \draw[color=black](6,3)rectangle(7,4); 
      \draw[color=black](6,4)rectangle(7,5); 
      \draw[color=black](6,5)rectangle(7,6); 
      \draw[color=black](6,6)rectangle(7,7); 
      \draw[color=black](6,7)rectangle(7,8); 
      \draw[color=black](6,8)rectangle(7,9);
      \draw[color=black](6,9)rectangle(7,10);  
       \draw[color=black](7,1)rectangle(8,2); 
      \draw[color=black](7,2)rectangle(8,3); 
      \draw[color=black](7,3)rectangle(8,4); 
      \draw[color=black](7,4)rectangle(8,5); 
      \draw[color=black](7,5)rectangle(8,6); 
      \draw[color=black](7,6)rectangle(8,7); 
      \draw[color=black](7,7)rectangle(8,8); 
      \draw[color=black](7,8)rectangle(8,9);
       \draw[color=black](7,9)rectangle(8,10);  
      \draw[color=black](8,1)rectangle(9,2); 
      \draw[color=black](8,2)rectangle(9,3); 
      \draw[color=black](8,3)rectangle(9,4); 
      \draw[color=black](8,4)rectangle(9,5); 
      \draw[color=black](8,5)rectangle(9,6); 
      \draw[color=black](8,6)rectangle(9,7); 
      \draw[color=black](8,7)rectangle(9,8); 
      \draw[color=black](8,8)rectangle(9,9); 
      \draw[color=black](8,9)rectangle(9,10); 
      \draw[color=black](4,5)rectangle(6,7); 
      \node at (0.5,9.5) {$0$};
      \node at (1.5,9.5) {$0$};
      \node at (0.5,7.5) {$0$};
     \node at (1.5,7.5) {$0$};
      \node at (0.5,8.5) {$0$};
      \node at (1.5,8.5) {$0$};
      \node at (1.5,4.5) {$1$};
       \node at (1.5,5.5) {$1$};
       \node at (5.5,4.5) {$3$};
       \node at (4.5,4.5) {$3$};
       \node at (3.5,8.5) {$1$};
       \node at (4.5,8.5) {$1$};
       \node at (3.5,7.5) {$1$};
       \node at (4.5,7.5) {$1$};
       \node at (5.5,7.5) {$1$};
       \node at (5.5,8.5) {$1$};
     \end{tikzpicture}.
       \]
\end{example}

Before proceeding, we recall one very useful lemma.

\begin{lemma}\label{lem:concatinate}\cite[Lemma 1.3.14]{Con93}
Let $I$ and $J$  be  homogeneous ideals of a polynomial ring over a field, and fix a term order $\sigma$.  With respect to $\sigma$, let  $\mathcal{F}$  be  a  Gr\"obner basis of $I$ and $\mathcal{G}$ a Gr\"obner basis of $J$.  Then $\mathcal{F} \cup \mathcal{G}$ is a Gr\"obner basis of $I+J$ if and only if for all $f \in \mathcal{F}$ and $g \in \mathcal{G}$ there exists $e \in I \cap J$ such that $LT(e) = LCM(LT(f),LT(g))$. \qed
\end{lemma}

We are now prepared to show that, under a suitable inductive hypothesis, there is a lower outside corner so that the CDG generators of $C_{y,I_w}$ are Gr\"obner.

\begin{lemma}\label{lem:GrobnerLink}
With notation as above, if $w \in S_n$ avoids obstructions of Types $1$, $2$, and $3$ and Conjecture \ref{conj:mainresult} holds for all permutations of smaller Coxeter length than that of $w$, then either $D_w = \Dom(w)$ or there is some lower outside corner $(i,j)$ of $D_w \setminus \Dom(w)$ corresponding to the variable $y = z_{i,j}$ so that the generators $\{q_1, \ldots, q_k, h_1, \ldots, h_\ell\}$ of $C_{y,I_w}$ form a Gr\"obner basis under any diagonal term order.
\end{lemma}
\begin{proof}
Fix a diagonal term order $\sigma$, and assume $D_w \neq \Dom(w)$.  By Lemma \ref{lem:GrobnerQs}, there is some lower outside corner of $D_w$ so that, with notation as above, there exists $\ell' \geq 0$ so that the generators $\{q_1, \ldots, q_k, h_1, \ldots, h_{\ell'} \}$ form a Gr\"obner basis for $Q_{y,I_w}$.  

By Corollary \ref{cor:GrobnerDel} and the inductive hypothesis, $\{h_1,\ldots, h_\ell\}$ is a diagonal Gr\"obner basis for $N_{y,I_w}$.  Then by Lemma \ref{lem:concatinate}, the generators $\{q_1, \ldots, q_k, h_1, \ldots, h_\ell\}$ form a diagonal Gr\"obner basis for $C_{y,I_w}$ if and only if for every $q_a$ and $h_b$ there exists some $f \in (q_1, \ldots, q_k, h_1, \ldots, h_{\ell'}) \cap (h_1, \ldots, h_\ell)$ satisfying $LT(f) = LCM(LT(q_a),LT(h_b))$.  

If $h_b \in (q_1, \ldots, q_k, h_1, \ldots, h_{\ell'})$ or $q_a \in (h_1, \ldots, h_\ell)$, the result follows from Lemma \ref{lem:concatinate}, and if $LCM(LT(q_a),LT(h_b)) = LT(q_a) \cdot LT(h_b)$, then we take $f = q_a \cdot h_b$.  Otherwise, there is some $(r,s) \in \Ess(w) \setminus \Dom(m)$ with $\rank_w(r,s) = \deg(h_b)-1$ corresponding to the variable $e = z_{r,s}$ weakly southeast of all of the variables involved in $h_b$.  Because $w$ has no Type $3$ obstruction, $e$ is not strictly northwest of $y$.  Suppose first that $y$ is strictly east and weakly north of $e$.

In this case, let $M'$ be the matrix consisting of the union of the columns determining $q_a$ and $h_b$ and the union of the rows determining $q_a$ and $h_b$.  We will next describe an auxiliary matrix $M$ formed from $M'$. (For an example of the construction, see Example \ref{ex:formM} below.)  Form a matrix $M$ from $M'$ as follows: First set to $0$ any entry whose row index is not one of the rows determining $q_a$ and whose column index is not one of the columns determining $h_b$.  Next, whenever a column of $M'$ contains a variable dividing the leading term of $q_a$ and a distinct variable dividing the leading term of $h_b$, duplicate that column and replace in one copy of the column the variables coming only from $h_b$ by $0$.  Whenever a row of $M'$ contains a variable dividing the leading term of $q_a$ and a distinct variable dividing the leading term of $h_b$, duplicate that row and replace in one copy of the row the variables coming only from $q_a$ by $0$.  

Now $M$ will be a $d \times d$ matrix where $d = \deg(LCM(LT(q_a),LT(h_b)))$ because it will have one row and one column for each monomial dividing $LT(q_a)$ and one each for every monomial dividing $LT(h_b)$ but not $LT(q_a)$.  

By expressing $\det(M)$ as a sum of products of the $\deg(q_a)$-minors from the rows of $M$ originating from the submatrix of $Z_w$ determining $q_a$ and the $(d-q_a)$-minors in the remaining rows, we see that $\det(M) \in (q_1, \ldots, q_k)$.  Similarly, by expressing $\det(M)$ as a sum of products of $\deg(h_b)$-minors in the column originating from the submatrix of $Z_w$ determining from $h_b$ and $(d-\deg(h_b))$-minors in the remaining columns, we have $\det(M) \in (h_1, \ldots, h_\ell)$.  It is because $y$ is strictly east and weakly north of $e$ that the rows determining $q_a$ that every $\deg(q_a)$-minor in the specified rows is an element of $(q_1, \ldots, q_k)$ and that every $\deg(h_b)$-minor in the specified column is an element of $(h_1, \ldots, h_\ell)$.

Next we will see that $LT(\det(M)) = LCM(LT(q_a),LT(h_b))$.  Call $\tilde{M}$ the submatrix of $M'$ whose entries are northwest of both $e$ and $y$.  Set $\mu_1$ to be the product of the terms of $\tilde{M}$ that divide either $LT(q_a)$ or $LT(h_b)$.  Call $\mu_2$ the leading term of the determinant of the submatrix of $M$ consisting of the rows and columns used to determine $h_b$ excluding those involving a divisor of $\mu_1$.  Similarly, define $\mu_3$ to be the leading term of the determinant of the submatrix of $M$ consisting of the rows and columns determining $q_a$ excluding those involving a divisor of $\mu_1$.  

Notice that $\mu_1 \cdot \mu_2 \cdot \mu_3 = LCM(LT(q_a),LT(h_b))$ and that this product is a term of $\det(M)$.  To see that every other term of $\det(M)$ is smaller under $\sigma$, notice that because $w$ has no obstruction of Type $1$, the submatrix of $M'$ whose entries are both northwest of $e$ and northwest of $y$ must have $0$'s only in full rows and full columns along the north and west sides.  

It will now be more convenient for us to work with $\widehat{M}$, obtained from $M$ by adding the doubled copies of rows and columns obtained in the transition from $M'$ to $M$ so that no variable appears more than once.  Note that $\det(\widehat{M}) = \det(M)$.  If some other term of $\det(\widehat{M})$ is larger than $\mu_1 \cdot \mu_2 \cdot \mu_3$, there must be some entries of $\widehat{M}$ dividing $\mu_1 \cdot \mu_2 \cdot \mu_3$ whose row indices we may permute to obtain a larger monomial.  We may assume that this permutation consists of one cycle.  If all entries divide either $\mu_1 \cdot \mu_2$ or $\mu_1\cdot \mu_3$, we would obtain a term of $q_a$ or $h_b$, respectively, that is strictly larger than its leading term, which also cannot be.  But the permutation cannot send any divisor of $\mu_3$ to the row of a divisor of $\mu_1$, all of which are $0$ in that column, or vice versa.  Hence, $\mu_1 \cdot \mu_2 \cdot \mu_3 = LT(\det(\widehat{M})) = LT(\det(M))$, as desired.

Finally, if $y$ is strictly south and weakly west of $e$, a parallel argument gives the result.
\end{proof}

\begin{example}\label{ex:formM}
Below we give an example of the construction of the matrices $M'$ and $M$.  Let $w = 5237164$, $y = z_{4,6}$, \[
q_a = \det \begin{bmatrix} 0& 0 & {\color{blue}\underline{z_{1,5}}}\\
{\color{blue}\underline{z_{2,2}}} & z_{2,3} & z_{2,5}\\
z_{3,2}& {\color{blue}\underline{z_{3,3}}} & z_{3,5}
 \end{bmatrix} \in Q_{y, I_w}, \mbox{ and } h_b = \det \begin{bmatrix} 
 0 & {\color{blue}\underline{z_{2,2}}} & z_{2,3} & z_{2,4} \\
0 & z_{4,2} & {\color{blue}\underline{z_{4,3}}} & z_{4,4} \\
{ \color{blue} \underline{z_{5,1}}} & z_{5,2} & z_{5,3} & z_{5,4} \\
  z_{6,1} & z_{6,2} & z_{6,3} & {\color{blue}\underline{z_{6,4}}} \\
 \end{bmatrix} \in N_{y,I_w},
 \] in which case $LT(q_a) = z_{1,5}z_{2,2}z_{3,3}$, $LT(h_b) = z_{2,2}z_{4,3}z_{5,1}z_{6,4}$, and \[ 
 LCM(LT(q_a),LT(h_b)) = z_{1,5}z_{2,2}z_{3,3}z_{4,3}z_{5,1}z_{6,4} = LT(\det(M)).
 \]  The variables dividing the leading terms appearing throughout this example are noted in blue and also underlined.  Then \[
 M' = \begin{bmatrix} 
 0 & 0 & 0 & 0 & {\color{blue}\underline{z_{1,5}}}\\
 0 & {\color{blue}\underline{z_{2,2}}} & z_{2,3} & z_{2,4} & z_{2,5}\\
  0 & z_{3,2} & {\color{blue}\underline{z_{3,3}}} & z_{3,4} & z_{3,5}\\
0 & z_{4,2} & {\color{blue}\underline{z_{4,3}}} & z_{4,4} & z_{4,5}\\
  {\color{blue}\underline{z_{5,1}}} & z_{5,2} & z_{5,3} & z_{5,4} & z_{5,5} \\
  z_{6,1} & z_{6,2} & z_{6,3} &  {\color{blue}\underline{z_{6,4}}} & z_{6,5} 
 \end{bmatrix} \mbox{  \hspace{0.5cm}   and  \hspace{0.5cm}   } M  = \begin{bmatrix} 
 0 & 0 & 0 & 0 & 0 & {\color{blue}\underline{z_{1,5}}}\\
 0 & {\color{blue}\underline{z_{2,2}}} & z_{2,3} & z_{2,3} & z_{2,4} & z_{2,5}\\
  0 & z_{3,2} & {\color{blue}\underline{z_{3,3}}} & {\color{blue}\underline{z_{3,3}}} & z_{3,4} & z_{3,5}\\
0 & z_{4,2} & 0 & {\color{blue}\underline{z_{4,3}}}  & z_{4,4} & 0\\
  {\color{blue}\underline{z_{5,1}}} & z_{5,2} & 0 & z_{5,3} & z_{5,4} &0 \\
  z_{6,1} & z_{6,2} & 0  & z_{6,3}& {\color{blue}\underline{z_{6,4}}}  & 0
 \end{bmatrix}.
 \] 

 By expressing $\det(M)$ as a sum of products of $3$-minors in the first $3$ rows of $M$ with $3$-minors in the final $3$ rows, we see $\det(M) \in Q_{y,I_w}$, and, by expressing $\det(M)$ as the sum of products of $4$-minors in columns $1$, $2$, $4$, and $5$ with $2$-minors in columns $3$ and $6$, we see $\det(M) \in N_{y,I_w}$.  Observe that \begin{align*}
 LT(\det(M)) &=  LT\left(\det \begin{bmatrix} z_{2,2} & z_{2,3} \\ z_{3,2} & z_{3,3} \end{bmatrix} \right) \cdot LT\left(\det  \begin{bmatrix} 0 & z_{4,3} & z_{4,4} \\ z_{5,1} & z_{5,3} & z_{5,4} \\ z_{6,1} & z_{6,3} & z_{6,4} \end{bmatrix}\right) \cdot z_{1,5} \cdot\\
 &= (z_{2,2}z_{3,3})(z_{4,3}z_{5,1}z_{6,4})(z_{1,5}).
 \end{align*}
 This expression corresponds to the product $\mu_1 \cdot \mu_2 \cdot \mu_3$ in the proof of Lemma \ref{lem:GrobnerLink}.  One may prefer to use $\widehat{M}$, as in the lemma, obtained from $M$ by subtracting column $3$ from column $4$, which has the effect of setting the copies of $z_{2,3}$ and $z_{3,3}$ in column $4$ to $0$.  
 
One aspect of this example that is typical of the general case is that the intersection of the submatrix of of $M'$ northwest of $z_{6,4}$ (playing the role of $e$) and that northwest of $z_{3,5}$ (from $z_{4,6}$ playing the role of $y$) is a rectangular matrix of indeterminates with $0$'s appearing only in full rows along the top and full columns along the western side of the submatrix.  This arrangement follows from the fact that $5237164$ has no obstruction of Type $1$ and gives rise to the decomposition of $LT(\det(M))$ described above into a product of the leading term of entries along the main diagonal of a matrix of indeterminates together with entries coming only from $q_a$ and entries coming only from $h_b$.
\end{example}

We are now prepared to use our lemmas above that establish Gr\"obner bases for $N_{y,I_w}$ and $C_{y,I_w}$ by induction to give the backward direction of Conjecture \ref{conj:mainresult}.  We recall a lemma that structures our proof below.  This lemma gives an implementation of the strategy employed throughout \cite{GMN13}.  

\begin{lemma}\label{lem:mainApp}\cite[Corollary 4.13]{KR}.  
Let $I = ( yq_1+r_1,\dots, yq_k+r_k,h_1,\dots, h_\ell )$ be a homogenous ideal of the polynomial ring $R$ with $y$ some variable of $R$ and $y$ not dividing any term of any $q_i$ nor any term of any $h_j$. Fix a term order $\sigma$ satisfying $LT(yq_i+r_i) = yLT(q_i)$ for each $i$.   Suppose that $\mathcal{G}_C = \{q_1,\dots, q_k,h_1,\dots, h_\ell\}$ and $\mathcal{G}_N = \{h_1,\dots, h_\ell\}$ are Gr\"obner bases for the ideals they generate, which we call $C$ and $N$, respectively, and that $\hgt(I)$, $\hgt(C)>\hgt(N)$.  Assume that $N$ has no embedded primes, and let $M = \begin{pmatrix}
q_1& \cdots & q_k\\
r_1& \cdots & r_k
\end{pmatrix}.$ If the ideal of $2$-minors of $M$ is contained in $N$, then the given generators of $I$ are a Gr\"obner basis.
\end{lemma}

\begin{theorem}\label{thm:if}
If $w \in S_n$ is a permutation that has no obstruction of Type $1$, Type  $2$, or Type $3$, then $w$ is CDG.
\end{theorem}

\begin{proof}
Fix a diagonal term order $\sigma$.  We proceed by induction on the Coxeter length of $w$, with the case of length $0$ trivial.  Fix a permutation $w \in S_n$ with $n$ arbitrary and assume that $w$ has no obstruction of Type $1$, Type  $2$, or Type $3$.  If $D_w = \Dom(w)$, then $I_w$ is generated by variables, and so the result is immediate.  Hence, we assume $\Dom(w) \neq D_w$.

According to Lemma \ref{lem:GrobnerLink}, there is some lower outside corner $(i,j)$ of $D_w \setminus \Dom(w)$ corresponding to the variable $y = z_{i,j}$ so that, with our usual notation, $\{yq_1+r_1,\dots, yq_k+r_k,h_1,\dots, h_\ell\}$ are the CDG generators of $I_w$, the generators $\{q_1, \ldots, q_k, h_1, \ldots, h_\ell\}$ of $C_{y,I_w}$ form a Gr\"obner basis, and $LT(yq_a+r_a) = yLT(q_a)$ for each $a$.  By Corollary \ref{cor:GrobnerDel} and the inductive hypothesis, $\{h_1, \ldots, h_\ell\}$ is a Gr\"obner basis for $N_{y,I_w}$.  

We will show that $I_2 \begin{pmatrix} q_1 & \ldots & q_k \\ r_1 & \ldots & r_k \end{pmatrix} \subseteq  N_{y,I}$.  For each CDG generator, $yq_a+r_a$, let $yq_a'+r_a'$ be the corresponding natural generator of $I_w$, i.e., the generator taken in a matrix of indeterminates in which the variables corresponding to $\Dom(w)$ have not been set to $0$.  Let $J = (z_{i,j} \mid (i,j) \in \Dom(w))$.  Then $I_2 \begin{pmatrix} q_1 & \ldots & q_k \\ r_1 & \ldots & r_k \end{pmatrix}+J= I_2 \begin{pmatrix} q_1' & \ldots & q_k' \\ r_1' & \ldots & r_k' \end{pmatrix}+J$.  Hence, because $J \subseteq N_{y,I_w}$, it suffices to show $I_2 \begin{pmatrix} q_1' & \ldots & q_k' \\ r_1' & \ldots & r_k' \end{pmatrix} \subseteq N_{y,I_w}$.  

In order to show this last containment, we will temporarily consider a possibly different diagonal term order $\sigma'$, which will be a lexicographic term order in which $y$ is largest.  It is because $y$ is a lower outside corner that there must exist a term order that is both diagonal and is also lexicographic with $y$ largest.  Now because each $q'_ar'_b-q'_br'_a = (yq'_a+r'_a)q'_b-(yq'_b+r'_b)q'_a$ for $1 \leq a<b \leq k$ is an element of the ideal of  $(\rank_w(i,j)+1)$-minors in a matrix of indeterminates weakly northwest of $y$, an ideal for which the natural generators form a Gr\"obner basis under $\sigma'$ because it is a diagonal order, we know that $q'_ar'_b-q'_br'_a$ has a Gr\"obner reduction in terms of those generators.  Because $q'_ar'_b-q'_br'_a$ does not involve $y$ and $y$ is lexicographically largest under $\sigma'$, that reduction must be in terms of $(\rank_w(i,j)+1)$-minors weakly northwest of $y$ that do not involve $y$.  Each such minor is an element of $N_{y,I_w}$.  Hence, \[
I_2 \begin{pmatrix} q_1 & \ldots & q_k \\ r_1 & \ldots & r_k \end{pmatrix}+J = I_2 \begin{pmatrix} q_1' & \ldots & q_k' \\ r_1' & \ldots & r_k' \end{pmatrix}+J \subseteq N_{y,I_w},
\] as desired. 

The height requirements $\hgt{I_w}, \hgt{C_{y,I_w}} > \hgt{N_{y,I_w}}$ are immediate from the fact that $N_{y,I_w}$ is prime \cite[Proposition 3.3]{Ful92} together with the proper containment of $N_{y,I_w}$ in each of $I_w$ and $C_{y,I_w}$.  The result now follows from Lemma \ref{lem:mainApp}.
\end{proof}

Notice that we do not claim that the Gr\"obner reduction of $q'_ar'_b-q'_br'_a$ with respect to $\sigma'$ gives rise to a Gr\"obner reduction of $a_ar_b-q_br_a$ in terms of the CDG generators with respect to $\sigma$.  Lemma \ref{lem:mainApp} requires only that we demonstrate an ideal containment.

\begin{corollary}\label{cor:mainIf}
If $w \in S_n$ avoids all eight of the following patterns, then $w$ is CDG: \[
13254, 21543, 214635, 215364, 215634, 241635, 315264, 4261735.
\]
\end{corollary}
\begin{proof}
If $w \in S_n$ avoids the patterns above, then it does not have an obstruction of Type $1$, Type  $2$, or Type $3$ by Lemmas \ref{lem:obstruction1}, \ref{lem:obstruction2}, and \ref{lem:obstruction3}, and so the result follows from Theorem \ref{thm:if}.
\end{proof}

\section{The non-CDG Permutations}\label{sect:forward}

In this section, we show that a permutation $w \in S_n$ that contains one of the eight permutations listed in Conjecture \ref{conj:mainresult} is not CDG.   We will show that slightly stronger claim that if $w \in S_n$ contains one of the eight listed patterns, then the CDG generators do not form a Gr\"obner basis under \emph{any} diagonal term order.  

Note that a generating set for an ideal $I$ in a polynomial ring $R$ forms a Gr\"obner basis for $I$ if and only if that generating set forms a Gr\"obner basis in the larger polynomial ring $R[x]$.   For that reason, we may view all ideals that arise in Theorem \ref{onlyif} as ideals of a polynomial ring in $(n+1)^2$ variables.  An example illustrating the argument of Theorem \ref{onlyif} follows immediately after the proof.

\begin{theorem}\label{onlyif}
Let $Z$ by an $(n+1) \times (n+1)$ matrix of indeterminates, and let $\sigma$ be a diagonal term order on $R=\mathbb{C}[Z]$.  Let $d \leq n$, and suppose that there is some $w \in S_d$ so that the CDG generators of $I_w$ do not form a Gr\"obner basis under $\sigma$.  If $v \in S_{n+1}$ contains $w$, then the CDG generators of $I_v$ do not form a  Gr\"obner basis under $\sigma$.
\end{theorem}
\begin{proof}
By induction, we may assume that $w = w_1 \ldots w_n \in S_n$ and that $v = v_1 \ldots v_{n+1}$ with $v_1 \ldots v_{i-1} v_{i+1} \ldots v_{n+1} = w$ for some $ 1 \leq i \leq n+1$.  

Recall that $D_w$ is obtained from $D_v$ by deleting row $i$ and column $v_i$.  With $Z_v$ an $(n+1) \times (n+1)$ matrix of indeterminates with $z_{i,j}$ set to $0$ whenever $(i,j) \in \Dom(v)$, identify $Z_w$ with the $n \times n$ submatrix of $Z_v$ obtained by the deletion of row $i$ and column $v_i$.   Consider that the rows of $Z_w$ to be labeled $1, \ldots, i-1, i+1, \ldots n+1$ and the columns of to be labeled $1, \ldots, v_{i-1},v_{i+1}, \ldots, n+1$.   

Let $G_w = \{\delta_1,\ldots, \delta_\ell\}$ for some $\ell \in \mathbb{N}$ be the set of CDG generators of $w$.  Assume that $G_w$ is ordered so that $\delta_1, \ldots, \delta_k$ are determined by rank conditions in boxes $(a,b) \in \Ess(w)$ with $a<i$ or $b<v_i$ and that $\delta_{k+1}, \ldots, \delta_\ell$ are determined by rank conditions in boxes $(a,b)$ with $a>i$ and $b>v_i$.  

Let $f$ and $g$ be two CDG generators of $I_w$ whose $s$-polynomial $s = s(f,g)$ does not reduce to $0$ by $G_w$ under $\sigma$.  Let $r$ denote the remainder of $s$ under the deterministic division algorithm with respect to $G_w$ and the chosen ordering on $G_w$.  Then we may write $r = s+\sum \alpha_j \delta_j$ where the leading term of $\alpha_j \delta_j$ is not in the ideal generated by the leading terms of the $\delta_{j'}$ with $j'<j$.  By definition of remainder, no leading term of any element of $G_w$ divides the leading term of $r$ though $r \in I_w$.

Let $G_v$ denote the set of CDG generators of $I_v$.  We may write \[
G_v = \{\delta_1, \ldots, \delta_k, z_{i,v_i}\delta_{k+1}+\varepsilon_{k+1}, \ldots, z_{i,v_i}\delta_{\ell}+\varepsilon_{\ell}, \delta_{\ell+1}, \ldots, \delta_{m}\},
\] where the $\delta_j$ with $\ell < j \leq m$ are the elements of $G_v$ involving at least one variable from row $i$ or column $v_i$ other than $z_{i,v_i}$, and the others are as expected.  We will use $r$ to construct an element of $I_v$ whose leading term is not divisible by any leading term of $G_v$.  

If the southeast corner of the submatrix of $Z_w$ determining $f$ is a box $(a,b)$ satisfying $a<i$ or $b<v_i$, then $f \in G_v$.  In that case, define $f' = f$.  If $a>i$ and $b>v_i$, then take $f'$ to be (up to sign) the element of $G_v$ determined by the rows determining $f$ together with row $i$ and the columns determining $f$ together with column $v_i$.  After possibly multiplying by $-1$, $f' = z_{i,v_i}f+\varepsilon_f$, where every term of $\varepsilon_f$ is divisible by exactly one variable from row $i$ and exactly one variable from column $v_i$, neither of which is $z_{i,v_i}$.  Define $g'$ similarly, and take $s' = s(f',g')$ to be their $s$-polynomial.  

If $f' = f$ and $g' = g$, then $s' = s$.  Because no term of $s$ is divisible by any variable in row $i$ or in column $v_i$ of $Z_v$, if $s$ has a reduction by the elements of $G_v$, it must have a reduction by $\{\delta_1, \ldots, \delta_k\}$, which is known not to exist.  

If $f' = z_{i,v_i}f+\varepsilon_f$ and $g' = g$, let $LT(f)$ denote the leading term of $f$, $LT(g)$ denote the leading term of $g$, and $G$ the greatest common divisor of $LT(f)$ and $LT(g)$.  Set \[
t = \frac{LT(g)}{G}f'-\frac{z_{i,v_i}LT(f)}{G}g = z_{i,v_i}s+\frac{LT(g)}{G}\varepsilon_f \in I_v.
\]  (Notice that whenever $z_{i,v_i}LT(f)$ is the leading term of $f'$, $t$ will coincide with the $s$-polynomial of $f'$ and $g'$.)  

We claim that $t$ cannot be reduced by $G_v$.  We begin by modifying $t$ by multiples of the $\delta_i$ for $1 \leq i \leq k$ following the deterministic division algorithm in $G_w$ to obtain $t' = z_{i,v_i}r+\frac{LT(g)}{G}\varepsilon_f \in I_v$.  Notice that $\frac{LT(g)}{G}\varepsilon_f$ does not involve $z_{i,v_i}$ and that every element of $G_v$ involving $z_{i,v_i}$ involves it only as a multiple of some $\delta_j$ with $k<j\leq \ell$.  Hence, no term of any $\delta_j$ divides any term of $z_{i,v_i}r$, and the division algorithm will never call for the addition of any multiple of any $z_{i,v_i}\delta_j$.  Therefore, no newly added polynomial could have any term that cancels with any term of $z_{i,v_i}r$, from which it follows that $t'$ is an element of $I_v$ with no reduction by $G_v$.  

Finally, assume that $f' = z_{i,v_i}f+\varepsilon_f$ and $g' = z_{i,v_i}g+\varepsilon_g$.  Then \[
t = \frac{LT(g)}{G}f'-\frac{LT(f)}{G}g' = s+\frac{LT(g)}{G}\varepsilon_f-\frac{LT(f)}{G}\varepsilon_g \in I_v.
\] Again, we modify by multiples of the $\delta_j$ with $1 \leq j \leq k$ to obtain $t' = r+\frac{LT(g)}{G}\varepsilon_f-\frac{LT(f)}{G}\varepsilon_g$.  Because no leading term of any $\delta_j$ with $1 \leq j \leq k$ divides any term of $r$ and because every term of every other element of $G_v$ involves a variable from row $i$ or column $v_i$, which $r$ does not, no further steps in the division algorithm can eliminate any term of $r$, and so $t$ has no reduction by the elements of $G_v$.  

It follows that in all cases, there is an element of $I_v$ that has no Gr\"obner reduction by $G_v$, and so $G_v$ is not a diagonal Gr\"obner basis of $I_v$.
\end{proof}

\begin{example}
Let $w = 21543$ and $v = 216354$, in which case $(i, v_i) = (4,3)$.  The visualizations of the Rothe diagrams of $w$ and $v$ are below with the modified row and column indexing for the diagram of $w$ described in the proof of Theorem \ref{onlyif}.  \[
 \begin{minipage}{.4\textwidth}
\hspace{0.5cm} 
  \begin{tikzpicture}[x=1.5em,y=1.5em]
      \draw[color=black, thick](0,1)rectangle(5,6);
           \node at (2.5,6.5) {$w = 21543$};
           \node at (-0.5,1.5) {$6$};
           \node at (-0.5,2.5) {$5$};
           \node at (-0.5,3.5) {$3$};
          \node at (-0.5,4.5) {$2$};
          \node at (-0.5,5.5) {$1$};
         \node at (4.5,0.5) {$6$};
           \node at (3.5,0.5) {$5$};
           \node at (2.5,0.5) {$4$};
          \node at (1.5,0.5) {$2$};
          \node at (0.5,0.5) {$1$};
      \draw[color=black](1,4)rectangle(2,5);
      \draw[color=black](3,4)rectangle(4,5);
       \draw[color=black](1,3)rectangle(2,4);
      \draw[color=black](3,3)rectangle(4,4);
      \draw[color=black](1,2)rectangle(2,3);      
     \draw[thick,color=red] (5,5.5)--(1.5,5.5)--(1.5,1);
     \draw[thick,color=red] (5,4.5)--(0.5,4.5)--(0.5,1);
     \draw[thick,color=red] (5,3.5)--(4.5,3.5)--(4.5,1);
     \draw[thick,color=red] (5,2.5)--(3.5,2.5)--(3.5,1);
     \draw[thick,color=red] (5,1.5)--(2.5,1.5)--(2.5,1);
     \filldraw [black](1.5,5.5)circle(.1);
     \filldraw [black](0.5,4.5)circle(.1);
      \filldraw [black](4.5,3.5)circle(.1);
     \filldraw [black](3.5,2.5)circle(.1);
     \filldraw [black](2.5,1.5)circle(.1);
      \draw[color=black](0,5)rectangle(1,6);
      \draw[color=black](1,5)rectangle(2,6);  
      \draw[color=black](2,5)rectangle(3,6);
      \draw[color=black](3,5)rectangle(4,6); 
      \draw[color=black](4,5)rectangle(5,6);
     \draw[color=black](0,4)rectangle(1,5); 
      \draw[color=black](2,4)rectangle(3,5);
      \draw[color=black](4,4)rectangle(5,5);
      \draw[color=black](0,3)rectangle(1,4); 
      \draw[color=black](2,3)rectangle(3,4);
      \draw[color=black](4,3)rectangle(5,4);
      \draw[color=black](0,2)rectangle(1,3); 
      \draw[color=black](2,2)rectangle(3,3);
      \draw[color=black](3,2)rectangle(4,3);
      \draw[color=black](4,2)rectangle(5,3);
      \draw[color=black](1,1)rectangle(2,2);  
      \draw[color=black](2,1)rectangle(3,2);
      \draw[color=black](3,1)rectangle(4,2); 
      \draw[color=black](4,1)rectangle(5,2);
     \end{tikzpicture}
  \end{minipage}    \begin{minipage}{.4\textwidth}
\hspace{1cm} 
  \begin{tikzpicture}[x=1.5em,y=1.5em]
      \draw[color=black, thick](0,1)rectangle(6,7);
           \node at (3,7.5) {$v=216354$};
           \node at (-0.5,1.5) {$6$};
           \node at (-0.5,2.5) {$5$};
          \node at (-0.5,3.5) {$4$};
           \node at (-0.5,4.5) {$3$};
          \node at (-0.5,5.5) {$2$};
          \node at (-0.5,6.5) {$1$};
          \node at (5.5,0.5) {$6$};
         \node at (4.5,0.5) {$5$};
           \node at (3.5,0.5) {$4$};
           \node at (2.5,0.5) {$3$};
          \node at (1.5,0.5) {$2$};
          \node at (0.5,0.5) {$1$};
     \draw[color=black](0,6)rectangle(1,7);
      \draw[color=black](1,6)rectangle(2,7);
      \draw[color=black](1,4)rectangle(2,5);
      \draw[color=black](3,4)rectangle(4,5);
       \draw[color=black](1,3)rectangle(2,4);
      \draw[color=black](3,3)rectangle(4,4);
      \draw[color=black](1,2)rectangle(2,3);      
     \draw[thick,color=red] (6,6.5)--(1.5,6.5)--(1.5,1);
     \draw[thick,color=red] (6,5.5)--(0.5,5.5)--(0.5,1);
     \draw[thick,color=red] (6,4.5)--(5.5,4.5)--(5.5,1);
     \draw[thick,color=red] (6,3.5)--(2.5,3.5)--(2.5,1);
     \draw[thick,color=red] (6,2.5)--(4.5,2.5)--(4.5,1);
     \draw[thick,color=red] (6,1.5)--(3.5,1.5)--(3.5,1);
     \filldraw [black](1.5,6.5)circle(.1);
     \filldraw [black](0.5,5.5)circle(.1);
      \filldraw [black](5.5,4.5)circle(.1);
     \filldraw [black](2.5,3.5)circle(.1);
     \filldraw [black](4.5,2.5)circle(.1);
      \filldraw [black](3.5,1.5)circle(.1);
      \draw[color=black](2,6)rectangle(3,7);
      \draw[color=black](3,6)rectangle(4,7); 
      \draw[color=black](4,6)rectangle(5,7);
      \draw[color=black](5,6)rectangle(6,7); 
      \draw[color=black](0,5)rectangle(1,6);
      \draw[color=black](1,5)rectangle(2,6);  
      \draw[color=black](2,5)rectangle(3,6);
      \draw[color=black](3,5)rectangle(4,6); 
      \draw[color=black](4,5)rectangle(5,6);
      \draw[color=black](5,5)rectangle(6,6); 
     \draw[color=black](0,4)rectangle(1,5); 
      \draw[color=black](2,4)rectangle(3,5);
      \draw[color=black](4,4)rectangle(5,5);
      \draw[color=black](5,4)rectangle(6,5); 
      \draw[color=black](0,3)rectangle(1,4); 
      \draw[color=black](2,3)rectangle(3,4);
      \draw[color=black](4,3)rectangle(5,4);
      \draw[color=black](5,3)rectangle(6,4); 
      \draw[color=black](0,2)rectangle(1,3); 
      \draw[color=black](2,2)rectangle(3,3);
      \draw[color=black](3,2)rectangle(4,3);
      \draw[color=black](4,2)rectangle(5,3);
      \draw[color=black](5,2)rectangle(6,3); 
      \draw[color=black](1,1)rectangle(2,2);  
      \draw[color=black](2,1)rectangle(3,2);
      \draw[color=black](3,1)rectangle(4,2); 
      \draw[color=black](4,1)rectangle(5,2);
      \draw[color=black](5,1)rectangle(6,2); 
     \end{tikzpicture}
  \end{minipage}   
       \]

Let $f = \begin{vmatrix} z_{2,1} & z_{2,2} & z_{2,4} \\
z_{3,1} & z_{3,2} & z_{3,4}\\
z_{5,1} & z_{5,2} & z_{5,4}
 \end{vmatrix}$ and $g = \begin{vmatrix} 0 & z_{1,4} & z_{1,5} \\
z_{2,1} & z_{2,4} & z_{2,5}\\
z_{3,1} & z_{3,4} & z_{3,5}
 \end{vmatrix}$, in which case, under any diagonal term order, $s = s(f,g)= z_{1,4}z_{3,5}f-z_{3,2}z_{5,4}g$ and \begin{align*}
r&= -z_{1,4}z_{2,2}z_{3,1}z_{3,5}z_{5,4}+z_{1,4}z_{2,2}z_{3,4}z_{3,5}z_{5,1}+ z_{1,4}z_{2,4}z_{3,1}z_{3,5}z_{5,2}\\
&-z_{1,4}z_{2,4}z_{3,2}z_{3,5}z_{5,1} +z_{1,4}z_{2,5}z_{3,1}z_{3,2}z_{5,4}-z_{1,4}z_{2,5}z_{3,1}z_{3,4}z_{5,2}\\
&+z_{1,5}z_{2,2}z_{3,1}z_{3,4}z_{5,4}-z_{1,5}z_{2,2}z_{3,4}^{2}z_{5,1}-z_{1,5}z_{2,4}z_{3,1}z_{3,2}z_{5,4} +z_{1,5}z_{2,4}z_{3,2}z_{3,4}z_{5,1}.
 \end{align*}

 Then $g' = g$ and $f' = \begin{vmatrix} z_{2,1} & z_{2,2} & z_{2,3} & z_{1,4}\\
z_{3,1} & z_{3,2} & z_{3,3} & z_{3,4}\\
z_{4,1} & z_{4,2} & z_{4,3} & z_{4,4}\\
z_{5,1} & z_{5,2} & z_{5,3} & z_{5,4}
 \end{vmatrix} = z_{4,3}f+\varepsilon_f$, where $\varepsilon_f$ consists of the terms of $f'$ arising as products of $z_{2,3}$, $z_{3,3}$, and $z_{5,3}$ with their respective cofactors (or, equivalently, products of $z_{4,1}$, $z_{4,2}$, and $z_{4,4}$ and their respective cofactors).  Set $t = z_{4,3}s+z_{1,4}z_{3,5}\varepsilon_f$ and $ t' = z_{4,3}r+(z_{1,4}z_{3,5}-z_{1,5}z_{3,4}) \varepsilon_f$.  The fact that $z_{4,3}r$ prevents the reduction of $t'$ to $0$ by $G_v$ follows from the fact that no term of $r$ is divisible by the leading term of any element of $G_w$.
\end{example}

\begin{corollary}\label{cor:mainOnlyif}
Let $w$ be a permutation. If there exists a diagonal term order $\sigma$ so that the CDG generators of $I_w$ form a Gr\"obner basis, then $w$ avoids all eight of the patterns \[
13254, 21543, 214635, 215364, 215634, 241635, 315264, 4261735.
\]
\end{corollary}
\begin{proof}
This result is immediate from Theorem \ref{onlyif} together with explicit computations in the case of the eight permutations listed in Conjecture \ref{conj:mainresult}.
\end{proof}

\section{Unifying characteristics of the non-CDG permutations}\label{sect:intuition}
We conclude by describing briefly how \cite[Theorem 2.1(a)]{KMY09} can be used to understand two properties that prevent the permutations listed in Conjecture \ref{conj:mainresult} from being CDG.  We note first that $13254$ has no dominant part and so its failure to be CDG is due to the fact that it contains $2143$ \cite[Theorem 6.1]{KMY09}.  For the remainder of this section, we consider the other seven permutations, all of which have nontrivial dominant parts.  

For an arbitrary rank matrix, understand the CDG generators to be defined analogously to the case of defining ideals of matrix Schubert varieties.  If any of the permutations listed in Conjecture \ref{conj:mainresult} were CDG, \cite[Theorem 2.1(a)]{KMY09} would require that either the ideal determined by the rank matrix $N_1$ or the ideal determined by the rank matrix $N_2$, below, have a CDG diagonal Gr\"obner basis, which they are easily seen not to:  \newline
\begin{minipage}{0.33\textwidth}\centering\[\hspace{2cm}
N_1 = \begin{bmatrix}
        0 & 1 & 1  \\
        1 & 1 & 1 \\
        1 & 2 &2  \\
        1 & 2 & 2  
        \end{bmatrix} 
\]\end{minipage}\begin{minipage}{.33\textwidth}
\centering and 
\end{minipage} \begin{minipage}{0.33\textwidth} \centering\[\hspace{-2cm}
N_2 = \begin{bmatrix}
        1 & 1 & 1  \\
        1 & 1 & 2  \\
        1 & 2 &2  \\
        \end{bmatrix}.
\] 
\end{minipage}

The rank matrix $N_1$ encodes interference from $\Dom(w)$ that prevents $I_w$ from being CDG, and $N_2$ encodes failures to be vexillary that are sufficiently far from $\Dom(w)$ that they are not handled by replacing Fulton generators by CDG generators. 

\begin{example} 
Consider the rank matrix $M_w$ of the permutation $w = 21543$ with respect to any lexicographic term order in which $y = z_{3,4}$ is largest, with essential boxes marked by ${\color{burntorange}\square}$.

\[
   M_w =  \begin{bmatrix}
        {\color{burntorange}\boxed{0}} & 1 & 1 & 1 & 1 \\
        1 & 2 & 2 & 2 &2 \\
        1 & 2 &2 &  {\color{burntorange}\boxed{2}}  & 3 \\
        1 & 2 &  {\color{burntorange}\boxed{2}}  & 3 &4 \\
        1 & 2 &3 & 4 & 5
\end{bmatrix}
\]

If the CDG generators of $I_w$ were a Gr\"obner basis, then \cite[Theorem 2.1(a)]{KMY09} would require the CDG generators of $C_{y,I_w}$, which is the ideal determined by $N_1$ and plays the role of the link in a geometric vertex decomposition at $y$, also to be a Gr\"obner basis.  
\end{example}

We leave it to the reader to use (possibly repeated) application of \cite[Theorem 2.1(a)]{KMY09} to obtain $N_1$ or $N_2$ from the rank matrices of the other six permutations listed in Conjecture \ref{conj:mainresult}.

\bibliographystyle{plain}

\end{document}